\documentclass[11pt, notitlepage]{article}
\usepackage{amssymb,amsmath,comment}
\catcode`\@=11 \@addtoreset{equation}{section}
\def\thesection{\arabic{section}}

\def\theequation{\thesection.\arabic{equation}}
\catcode`\@=12
\usepackage{colortbl}%
\usepackage{a4wide}
\usepackage{parskip}

\newcommand{\ds} {\displaystyle}
\newcommand{\e}{\epsilon}

\newcommand{\al} {\alpha}
\newcommand{\ba} {\beta}
\newcommand{\sg} {\sigma}
\newcommand{\de} {\delta}
\newcommand{\ga} {\gamma}

\newcommand{\Om} {\Omega}
\newcommand{\ra} {\rightarrow}

\newcommand{\De} {\Delta}
\newcommand{\la} {\lambda}
\newcommand{\La} {\Lambda}
\newcommand{\noi} {\noindent}
\newcommand{\na} {\nabla}

\newcommand{\oline} {\overline}
\newcommand{\mb} {\mathbb}
\newcommand{\mc} {\mathcal}
\newcommand{\lra} {\longrightarrow}

\usepackage[all]{xy}
\catcode`\@=11
\def\theequation{\@arabic{\c@section}.\@arabic{\c@equation}}
\catcode`\@=12

\def\QED{\hfill {$\square$}\goodbreak \medskip}

\newtheorem{Theorem}{Theorem}[section]
\newtheorem{Lemma}[Theorem]{Lemma}
\newtheorem{Proposition}[Theorem]{Proposition}

\newtheorem{Remark}[Theorem]{Remark}
\newtheorem{Definition}[Theorem]{Definition}

\def\Xint#1{\mathchoice
	{\XXint\displaystyle\textstyle{#1}}%
	{\XXint\textstyle\scriptstyle{#1}}%
	{\XXint\scriptstyle\scriptscriptstyle{#1}}%
	{\XXint\scriptscriptstyle\scriptscriptstyle{#1}}%
	\!\int}
\def\XXint#1#2#3{{\setbox0=\hbox{$#1{#2#3}{\int}$ }
		\vcenter{\hbox{$#2#3$ }}\kern-.6\wd0}}

\begin{document}
{\vspace{0.01in}
\title
{Regularity and multiplicity results for fractional
	$(p,q)$-Laplacian equations}

\author{ {\bf Divya Goel\footnote{e-mail: divyagoel2511@gmail.com}\;, Deepak Kumar\footnote{email: deepak.kr0894@gmail.com}  \;and  K. Sreenadh\footnote{	e-mail: sreenadh@maths.iitd.ac.in}} \\ Department of Mathematics,\\ Indian Institute of Technology Delhi,\\
	Hauz Khaz, New Delhi-110016, India. }
\date{}

\maketitle

\begin{abstract}
 \noi This article deals with the study of the  following nonlinear  doubly nonlocal   equation: 
  \begin{equation*}
   (-\Delta)^{s_1}_{p}u+\ba(-\Delta)^{s_2}_{q}u  = \la a(x)|u|^{\delta-2}u+ b(x)|u|^{r-2} u,\;  \text{ in }\; \Om, \; u=0 \text{ on } \mathbb{R}^n\setminus \Om,
  \end{equation*}
  where $\Om$ is a bounded domain in $\mathbb{R}^n$ with smooth boundary, $1< \de \le q\leq p<r \leq p^{*}_{s_1}$, with $p^{*}_{s_1}=\ds \frac{np}{n-ps_1}$, $0<s_2 < s_1<1$, $n> p s_1$ and $\la, \ba>0$ are  parameters. Here $a\in L^{\frac{r}{r-\de}}(\Om)$ and $b\in L^{\infty}(\Om)$ are sign changing functions. We prove the $L^\infty$ estimates, weak Harnack inequality and Interior H\"older regularity of the  weak solutions of the above problem in the subcritical case $(r<p_{s_1}^*).$ Also, by analyzing the fibering maps and minimizing the energy functional over suitable
  subsets of the Nehari manifold, we prove existence and multiplicity of weak solutions to above 
  convex-concave problem. In case of $\de=q$, we show the existence of  solution.   

\medskip

 \noi \textbf{Key words:} Doubly nonlocal equation, Fractional
 $(p,q)$-Laplacian,  Nehari manifold, Regularity results.

\medskip

\noi \textit{2010 Mathematics Subject Classification:} 35J35, 35J60, 35R11

\end{abstract}

\bigskip
\vfill\eject

\section {Introduction}
  \noindent The purpose of this article is to study the existence and multiplicity of solutions of the following fractional problem
    \begin{equation*}
        (\mc P_\la) \; \left\{
         \quad  (-\Delta)^{s_1}_{p}u+ \ba (-\Delta)^{s_2}_{q}u = \la a(x)|u|^{\de-2}u+ b(x)|u|^{r-2} u, \; \text{in} \; \Om,   u =0, \text{on} \; \mb R^n\setminus \Om,
          \right.
      \end{equation*}
  \noi where $\Om$ is a bounded domain in $\mb R^n$ with smooth boundary, $1< \de\le q\leq p<r \leq p^{*}_{s_1}$, with $p^{*}_{s_1}=\ds \frac{np}{n-p s_1}, 0< s_2 < s_1<1, n> p s_1$ and $\la,\ba>0$ are real parameters. Here $a\in L^{\frac{r}{r-\de}}(\Om)$ and $b\in L^{\infty}(\Om)$ are sign changing functions and the fractional $p$-Laplace operator $(-\Delta)^{s}_{p}$ is defined as 
       \begin{equation*}
            {(-\Delta)^{s}_pu(x)}= 2\ds\lim_{\e\ra 0}\int_{\mathbb R^n\setminus B_\e(x)}\frac{|u(y)-u(x)|^{p-2}(u(y)-u(x))}{|x-y|^{n+ps}}dy.
       \end{equation*}
  In recent years there has been an ample amount of work on  nonlocal operators, particularly on fractional $p$-Laplacian due to its wide applications in real world such as finance, obstacle problems, conservation laws, phase transition, image processing, anomalous diffusion, material science and many more. For more details, we refer to \cite{bakunin, caffarelli, gilboa, levendorski, nezzaH, pham, silvestre, zhu}  and the references therein. Problems of the type $(\mc P_\la)$  are known as double phase equations where the leading operator switches between two nonlocal nonlinear operators. \par
 In the local case, the problems of the type $(\mc P_\la)$ are known as $(p,q)$-Laplacian problem 
     \begin{equation}\label{prb3}
        \quad  -\Delta_{p}u- \ba\Delta_{q}u = f(x,u) \quad \text{in } \; \Om, \qquad  u =0\quad \text{on} \ \mb R^n\setminus \Om,
     \end{equation}
  where $\De_p u:= \na \cdot (|\na u|^{p-2}\na u)$. These equations arise while studying the stationary solutions of general reaction-diffusion equation
   \begin{equation}\label{prb1}
	   u_t= \na \cdot [A(u)\nabla u]+ r(x,u),
   \end{equation}
  where $A(u)= |\nabla u|^{p-2}+|\nabla u|^{q-2}.$ The 
  problem \eqref{prb1} has  applications in biophysics, plasma physics and chemical reactions, where the function $u$ corresponds to the concentration term, the first term on the right hand side represents diffusion with a diffusion coefficient $A(u)$ and the second term is the reaction which relates to sources and loss processes. In such applications, the reaction term $r(x,u)$ has polynomial form with respect to the concentration $u$. For more details,  readers are referred to   \cite{benci, cherfils, jia}. 
  This   wide range of applications,  provoked many researchers to study stationary solutions of \eqref{prb1} that is,  \begin{equation}\label{prb2}
       -\na\cdot[A(u)\nabla u]=r(x,u),\end{equation} 
  for instance, Marano and Papageorgiou \cite{marano} obtained three solutions of the problem \eqref{prb3} using variational methods and truncation arguments. While Li and Zhang \cite{li} studied stationary solutions of problem \eqref{prb1} with concave-convex nonlinearities by taking $r(x,u)= |u|^{p^*-2}u+\theta |u|^{r-2}u, 1<r<p^*,\; p^*=\frac{np}{n-p}$, $1< p<n$ and using Lusternik-Schnirelman theory they proved infinitely many weak solutions of the problem in $W^{1,p}_0(\Om)$ for some range of $\theta$. Moreover, Yin and Yang \cite{yin} considered Dirichlet problem corresponding to \eqref{prb2} with $r(x,u)=|u|^{p^*-2}u+\theta V(x)|u|^{r-2}u+\la f(x,u)$, where $f(x,u)$ is a subcritical perturbation, and proved multiplicity of solutions.
  Recently Huang et al. \cite{huang} considered problem \eqref{prb2} in $\mb R^n$ to prove multiplicity of solutions with the reaction term $r(x,u)=K(x)|u|^{p^*-2}u+\la V(x)|u|^{k-2}u$, when $1<k<q<p<n.$ We refer \cite{maranorcn} for a survey of some recent advances in $(p,q)-$Laplacian problems.\par
  When $\ba=0$ or $p=q$ and $s_1=s_2$, the problem in $\mc (\mc P_\la)$ reduces to the following:
      \begin{equation}\label{prb4}
          \quad  (-\Delta)_p^{s}u= \la\; a(x)|u|^{\de-2}u+ b(x)|u|^{r-2} u \; \text{in}\;
          \Om, \;
          u =0\; \text{on} \ \mb R^n\setminus \Om,
      \end{equation}
  which is a nonlocal elliptic equation involving fractional $p-$Laplacian with combination of concave and convex nonlinearities.
  For $p=2$, Br\"andle et al. \cite{brandle} studied problem \eqref{prb4} for the subcritical case where they proved existence and non-existence of non-negative solutions for some range of $\la$. In \cite{brrios} Barrios et al. proved existence and multiplicity of solutions for the nonhomogeneous fractional Laplacian equation, and Colorado et al. \cite{colorado} considered the problem involving the critical Sobolev exponent. Moreover, with general nonlinearity Wei and Su \cite{wei} proved existence and multiplicity using Mountain Pass Theorem.  
  For general $p$, Goyal and Sreenadh \cite{goyal} studied \eqref{prb4} in the subcritical case by minimizing the energy functional over some subsets of the Nehari manifold associated to this problem, and Chen and Deng \cite{chen} studied multiplicity of solutions of above problem by Nehari manifold and fibering maps for the critical case.\par
   Recently, Bhakta and Mukherjee \cite{bhakta} studied the following problem in bounded domain
  \begin{equation*}
  (F_{\theta,\la}) \left\{
  (-\Delta)^{s_1}_{p}u+ (-\Delta)^{s_2}_{q}u = |u|^{p_{s_1}^*-2}u+\theta V(x)|u|^{r-2}u+ \la f(x,u) \; \text{in}\; \Om, \;
  u =0\; \text{on} \ \mb R^n\setminus \Om,
  \right.
  \end{equation*}
  where $0<s_2<s_1<1$, $1<r<q<p<\frac{n}{s_1}$, and  $V$ and $f$ are some appropriate functions. Here  they proved $(F_{\theta,\la})$ has infinitely many weak solutions for some range of $\la$ and $\theta$. Moreover, for  $V(x)\equiv 1$, $\la=0$ and assuming certain other conditions on $n, q, r$, they proved the existence of $cat_\Om(\Om)$ many solutions of $(F_{\theta,\la})$ using Lusternik-Schnirelmann category theory. Regarding the $(p,q)$-fractional elliptic problems on whole domain $\mathbb{R}^n$, we  cite \cite{ambrosio,chenc}. 
  
  In \cite{Iannizzotto} Iannizzotto et al. studied the following equation  \begin{equation}\label{prb5}
  		       (-\De)_p^s\; u=f(x,u) \; \mbox{in } \Om, \; u=0 \; \mbox{in } \mb  R^n\setminus\Om, 
     \end{equation} 
  where $f$ satisfies the growth condition $|f(x,t)|\le a(|t|^{q-1}+|t|^{r-1})$ a.e. $x\in\Om$ and all $t\in\mb R$, and $1\le q\le p\le r< p^*_s$ to get existence of weak solutions by Morse theory. They proved weak solutions belong to $L^\infty(\Om)$, with additional assumption $1+\frac{q}{p}>\frac{r}{p}+\frac{r}{p_s^*}$, using fractional version of De Giorgi's iterations technique. As far as H\"older regularity is concerned, 
  Kassmann \cite{Kassmann} proved nonlocal Harnack inequality and interior H\"older estimates for $p=2$ and for general $p$,
   Di Castro et al. \cite{castro} extended De Giorgi-Nash-Moser iteration technique to obtain interior H\"older regularity of the following equation
      \begin{align*}
	      (-\De)_p^s u=0 \; \mbox{in } \Om, \; u=g \quad \mbox{in } \mb  R^n\setminus\Om.
      \end{align*}
  Here authors used Caccioppoli type theorem and logarithmic estimates with nonlocal tails for weak solution $u$. Moreover, Iannizzotto et al. \cite{iann} obtained global H\"older regularity for weak solutions of \eqref{prb5} when the term on the right hand side depends only on $x$ and is in  $L^\infty(\Om)$ by using barrier arguments and Krylov's approach to boundary regularity. Inspired from all these works, in this paper  we study the $L^\infty$ estimates and interior Holder regularity of weak solutions of $(\mc P_\la)$.  By using the fractional version of De Giorgi iteration technique, we prove the $L^\infty$ estimates for the weak solutions of $(\mc P_\la)$. We further prove the weak Harnack inequality for weak solutions of $(\mc P_\la)$  and  the interior H\"older regularity using the Moser iteration technique. Our main result on regularity is the following:
 \begin{Theorem} \label{thm1.1}
	Let $r<p_{s_1}^{*} $ and $u$ be a weak solution of $(\mc P_\la)$, then $u\in L^\infty(\Om).$ Moreover, if $q\ge 2$ and $a\in L^\infty(\Om)$, then there exists $\al\in(0,1)$ such that $u \in L^\infty(\Om) \cap C^{0,\al}(\oline{\Om^\prime}) $ for all $\Om^\prime \Subset \Om$.    
 \end{Theorem}

  Regarding the existence and multiplicity of solutions of problem $(\mc P_\la)$, we prove existence of two non-trivial non-negative solutions of problem $(\mc P_\la)$ in the subcritical case for all $\ba >0$ and for small $\la$.  For the critical case we restricted ourselves into the case when the function $a(x)$ is continuous and $b(x)\equiv 1$ in $\Om$.  By using the fibering map analysis and minimizing the energy functional over some subsets of the Nehari manifol, we prove the existence of  at least two nontrivial non-negative solutions of $(\mc P_\la)$ provided $\la$ and $\ba $ small enough.   In the critical case the energy functional fails to satisfy the Palais-Smale condition globally, so in Lemma \ref{lemm3} we show that the functional satisfy the  Palais-Smale condition  below the first critical level.  The main difficulty in this case is that we do not have the explicit form of minimizers for $S$ which we overcome by using some optimal asymptotic estimates for $S$ provided in \cite{mosconi}. When $\de=q$ and $r\leq p^*_{s_1}$, we prove existence of one solution of $(\mc P_\la)$ using Mountain Pass Theorem. We remark that, to the best of our knowledge the critical case results are new even in the local case $(s_1=s_2=1)$.
 The main results of the paper regarding the existence and multiplicity of solutions are the following:
 \begin{Theorem}\label{pqthm1}
	Let $r < p^{*}_{s_1}$. Then there exists a constant $\la_0>0$ such that  for all $\ba>0$ and $\la \in(0, \la_0)$, $(\mc P_\la)$ has at least two non-negative non-trivial solutions.
 \end{Theorem}
  
 \begin{Theorem}\label{pqthm2}
	For $r=p^{*}_{s_1}$, let the function $a(x)$ be continuous in $\Om$ and $b(x)\equiv 1$ in $\Om$, then there exist positive constants $\La_0, \La_{00}$ and $\ba_{00} $ such that  
	\begin{enumerate}
		\item [(i)]
	for any $\la \in(0, \La_0)$ and $\ba>0$, $(\mc P_\la)$ admits at least one non-negative non-trivial solution. 
   \item [(ii)] For  any  $\la \in(0, \La_{00})$ and $\ba\in(0,\ba_{00})$, $(\mc P_\la)$  admits at least two non-negative non-trivial solutions, provided $1<\frac{n(p-1)}{n-ps_1}\le\de<q$. 
	\end{enumerate}
\end{Theorem}}

	\begin{Theorem}\label{pqthm3}
		Let $\de=q$ and assume functions $a(x)$ and $b(x)$ are continuous in $\Om$, then there exist positive constants $\la_*, \La_{*}$ and $\ba_{*}$ such that
		\begin{itemize}
			\item[(i)] for any $\la\in(0,\la_*)$, the problem $(\mc P_\la)$ has at least one solution, provided $r<p^*_{s_1}$.
			\item[(ii)] For any $\la\in (0,\La_{*})$ and $\ba\in (0,\ba_{*})$, $(\mc P_\la)$ has at least one solution, provided $1<\frac{n(p-1)}{n-ps_1}\le q\le p< r=p^*_{s_1}$.
		\end{itemize}
	\end{Theorem}
   \noindent Turning to the layout of the paper: In Section 2, we study the regularity results and prove interior H\"older regularity of weak solutions of $(\mc P_\la)$. In Section 3, we study the fibering maps and Nehari manifold associated to the problem $(\mc P_\la)$. We give fibering map analysis and prove some technical results. In Section 4, we give the proof of  Theorem \ref{pqthm1} and \ref{pqthm2}.  In Section 5, we give the proof of Theorem \ref{pqthm3}.

\section{Regularity results}
\setcounter{equation}{0}
In this section we give the variational frame work and interior regularity results for solutions of $(\mc P_\la)$. 
  For any open subset $\Om$ of $\mathbb{R}^n, 0<s<1$ and $1<p<\infty$, the fractional Sobolev space is defined as 
      \begin{align*}
          W^{s,p}(\Om)= \left\lbrace u \in L^p(\Om): \int_{\Om}\int_{\Om} \frac{|u(x)-u(y)|^p}{|x-y|^{n+sp}}~dxdy <+ \infty \right\rbrace
      \end{align*}
  \noi endowed with the norm 
      \begin{align*}
          \|u\|_{W^{s,p}(\Om)}:=  \|u\|_{L^p(\Om)}+\left(\int_{\Om}\int_{\Om} \frac{|u(x)-u(y)|^p}{|x-y|^{n+sp}}~dxdy  \right)^{\frac{1}{p}}.
      \end{align*}
  \noi Consider the space 
        \begin{align*}
             \tilde{X}_{p,s_1}:= \bigg \{u:\mathbb{R}^n\ra \mathbb{R} \text{ is measurable }: u \in L^{p}(\Om), \int_{Q} \frac{|u(x)-u(y)|^{p}}{|x-y|^{n+ps_1}}~dxdy < \infty \bigg \},
         \end{align*}
  \noi where $Q= \mathbb{R}^n \setminus (\Om^c\times \Om^c)$ and $0<s_1<1$, then $\tilde{X}_{p,s_1}$ is a Banach space with the norm 
         \begin{align*}
            \|u\|_{\tilde{X}_{p,s_1}}:=  \|u\|_{L^p(\Om)}+\left(\int_{Q} \frac{|u(x)-u(y)|^{p}}{|x-y|^{n+ps_1}}~dxdy  \right)^{\frac{1}{p}}. 
         \end{align*}
  \noi Let $X_{p,s_1}$ denotes the closure of $C_c^{\infty}(\Om)$ in $\tilde{X}_{p,s_1}$, then  $X_{p,s_1}$ is a uniformly convex Banach space with  norm \big(equivalent to $\|\cdot\|_{\tilde{X}_{p,s_1}}\big)$ 
         \begin{align}\label{eqb1}
              \|u\|_{X_{p,s_1}}:=  \left(\int_{Q} \frac{|u(x)-u(y)|^{p}}{|x-y|^{n+ps_1}}~dxdy  \right)^{\frac{1}{p}}.
         \end{align}
  \noi Notice that the integral in \eqref{eqb1}  can be extended to $\mathbb{R}^{2n}$ as $u=0$ a.e. on $\mathbb{R}^n\setminus \Om$. For the sake of completeness, we give the proof of the following technical result. The proof below is an adaptation of \cite[Lemma 2.6]{cheeger}.
 \begin{Lemma}\label{cheeg}
	Let $1<q\leq p<\infty$ and $0<s_2<s_1 <1$, then there exists a constant  $C=C(|\Om|,\;n,\; p,\;q,\;s_1,\;s_2)>0$ such that 
	\begin{align*}
	\|u\|_{X_{q,s_2}}\leq C \|u\|_{X_{p,s_1}}, \; \; \text{ for all }\; u \in X_{p,s_1}.
	\end{align*}
 \end{Lemma}
 \begin{proof}
   It is enough to prove the lemma  for all $ u \in C_c^{\infty}(\Om).$   So, let $u \in C_c^{\infty}(\Om)$ then 
	 \begin{align*}
    	\|u\|_{X_{q,s_2}}^{q}= \int_{\mathbb{R}^n} \int_{\mathbb{R}^n}\frac{|u(x)-u(y)|^{q}}{|x-y|^{n+qs_2 }}~dxdy &=\int_{\mathbb{R}^n} \int_{\mathbb{R}^n}\frac{|u(x+h)-u(x)|^{q}}{|h|^{n+qs_2 }}~ dxdh \\
	    &:= I_1+I_2,
	 \end{align*}
  where 
	  \begin{align*}
         I_1= \ds \int_{|h|>1} \int_{\mathbb{R}^n}\frac{|u(x+h)-u(x)|^{q}}{|h|^{n+qs_2}}~dxdh \text{ and }I_2= \ds \int_{|h|\leq 1} \int_{\mathbb{R}^n}\frac{|u(x+h)-u(x)|^{q}}{|h|^{n+qs_2}}~dxdh.	
      \end{align*}
  Using H\"older's inequality, we obtain 
	  \begin{align*}
	    I_1 = \int_{|h|>1} \int_{\mathbb{R}^n}\frac{|u(x+h)-u(x)|^{q}}{|h|^{n+qs_2}}~dxdh 
	    & \leq C(q) \int_{|h|>1}\int_{\mathbb{R}^n}\frac{|u(x)|^{q}}{|h|^{n+qs_2}}~dxdh \\
	    & \leq C(q) \int_{|h|>1}\frac{dh}{{|h|^{n+qs_2}}}\int_{\Om}|u(x)|^{q}~dx\\
	    & \leq C(q,s_2) \int_{\Om}|u(x)|^{q}~dx \\
	    & = C(q,s_2,|\Om|) \|u\|_{L^{p}(\Om)}^{q}.
	\end{align*}
  Now, consider 
	 \begin{align*}
    	I_2=\ds \int_{|h|\leq 1}\int_{\mathbb{R}^n}\frac{|u(x+h)-u(x)|^{q}}{|h|^{n+qs_2}}~dxdh = \ds \int_{|h|\leq 1}\frac{dh}{|h|^{n-(s_1-s_2)q}} \int_{\mathbb{R}^n}\frac{|u(x+h)-u(x)|^{q}}{|h|^{qs_1}}~dx. \end{align*}
  Since $u \in C_c^{\infty}(\Om)$, it implies that for every $h$ with $|h|<1$, $u(x+h)-u(x)$ has compact support. As a result,
	 \begin{align*}
    	\int_{\mathbb{R}^n}\frac{|u(x+h)-u(x)|^{q}}{|h|^{qs_1}}~dx &
    	\leq C(|\Om|,p,q)\left(\int_{\mathbb{R}^n}\frac{|u(x+h)-u(x)|^{p}}{|h|^{s_1 p}}~dx \right)^{\frac{q}{p}}\\
	    & \leq C(|\Om|,p,q) \sup_{h>0}\left(\int_{\mathbb{R}^n}\frac{|u(x+h)-u(x)|^{p}}{|h|^{ps_1}}~dx \right)^{\frac{q}{p}}\\
	   & \leq C(|\Om|,p,q,n)(1-s_1)^{\frac{q}{p}}\|u\|_{X_{p,s_1}}^{q},
	\end{align*} 
  where the last inequality holds using \cite[Lemma A.1]{cheeger}. Thus,
	  \begin{align*}
	    I_2\leq C(|\Om|,p,q,n,s_1)\|u\|_{X_{p,s_1}}^{q} \ds \int_{|h|\leq 1}\frac{dh}{|h|^{n-(s_1-s_2)q}} \leq  C(|\Om|,p,q,n,s_1,s_2)\|u\|_{X_{p,s_1}}^{q}.
	  \end{align*}
  \noi Therefore using Poincar\'e inequality, we deduce that  
	  \begin{align*}
	    \|u\|_{X_{q,s_2}}^{q}\leq  C(|\Om|,p,q,n,s_1,s_2)\left( \|u\|_{L^p(\Om)}^{q}+\|u\|_{X_{p,s_1}}^{q}\right)\leq  C(|\Om|,p,q,n,s_1,s_2)\|u\|_{X_{p,s_1}}^{q}.
	  \end{align*}  
	\QED
\end{proof}
  \textbf{Notations:} For our convenience, we denote 
  \begin{align*}
     A_p(u,v)=\ds \int_{Q}\frac{|u(x)-u(y)|^{p-2}(u(x)-u(y))(v(x)-v(y))}{|x-y|^{n+p s_1}}~ dxdy,\;\text{ for all } u,v \in X_{p,s_1}, \text{ and}\\
    A_q(u,v)=\ds \int_{Q}\frac{|u(x)-u(y)|^{q-2}(u(x)-u(y))(v(x)-v(y))}{|x-y|^{n+q s_2}}~ dxdy,\;\text{ for all } u,v \in X_{q,s_2}. 
   \end{align*}
 From \cite{goyal}, we have continuous embedding of $X_{p,s_1}$ into $L^m(\Om)$ for $1\le m\le p^*_{s_1}$, therefore we define
    \begin{equation*}
    	S_m=\inf_{u\in X_{p,s_1}\setminus\{0\}} \frac{\|u\|_{X_{p,s_1}}^{p}}{\|u\|_{L^{m}(\Om)}^{p}}.
    \end{equation*}
  For the sake of convenience, we denote $S_{p^*_{s_1}} =S$.

 \begin{Definition}
	A function $u \in X_{p,s_1}$ is said to be  a solution of $(\mc P_\la)$ if $u$ satisfies 
	\begin{align*}
	 A_p(u,v)+ \ba A_q(u,v) -\la\int_{\Om}a(x)|u|^{\de-2}uv~dx - \int_{\Om}b(x)|u|^{r-2}uv~dx=0 \; \text{for all}\; v \in X_{p,s_1}.
	\end{align*}
 \end{Definition}
 The   Euler functional $\mc {J}_\la: X_{p,s_1} \ra\mb R$ associated to the problem $(\mc P_\la)$ is given by
 \[\mc {J}_\la(u)= \frac{1}{p}\|u\|_{X_{p,s_1}}^{p}+ \frac{\ba}{q}\|u\|_{X_{q,s_2}}^{q} -\frac{\la}{\de}\int_{\Om}a(x)|u|^{\de} ~dx -\frac{1}{r}\int_{\Om}b(x)|u|^{r} ~dx.\]
Now we prove the following $L^\infty$ estimate for weak solutions of $(\mc P_\la)$ in the subcritical case. 
 \begin{Theorem}\label{linfty}
	Let $r<p^{*}_{s_1}$ and $u$ be a weak solution of the problem $(\mc P_\la)$,
	then $u \in L^{\infty}(\Om).$
 \end{Theorem}
 \begin{proof}
  	Let $u$ be a weak solution of $(\mc P_\la)$ such that $u^+\neq 0$. Let $\rho\ge max\{1,\|u\|_{L^r(\Om)}^{-1}\}$ and $v=(\rho\|u\|_{L^r(\Om)})^{-1}u$. Then, $v\in X_{p,s_1}$ and $\|v\|_{L^r(\Om)}=\rho^{-1}$.
	For all $k\in\mb N$, we define $w_k=(v-1+2^{-k})^{+}$ and $w_0=v^+$, so $w_k\in X_{p,s_1}$. Moreover,
	   \begin{align*}
	     0\le w_{k+1}(x)\le w_k(x) \text{ a.e. in } \Om \ \text{ for all } k\in\mb N \text{ and } w_k(x)\ra (v(x)-1)^+ \ a.e. \ in \ \Om. \end{align*}
   Set $U_k=\|w_k\|^r_{L^r(\Om)}$, then 
	    $ \{w_{k+1}>0\}\subseteq \{w_k> 2^{-(k+1)}\}, \ v(x)<(2^{k+1}-1)w_k(x) \ \text{ for all } x\in\{w_{k+1}>0\} $  and from Chebyshev inequality, we have
	       \[|\{w_{k+1}>0\}| \leq 2^{(k+1)r}U_k.\]
   Using the inequality  $$|\xi^+-\eta^+|^l\leq |\xi-\eta|^{p-2}(\xi-\eta)(\xi^+-\eta^+),\ \text{for } \xi, \eta\in\mb R \text{ and } l>1,$$
	it follows that
	\[
	  \|w_{k+1}\|_{X_{p,s_1}}^{p}= \ds\int_{\mb R^{2n}} \frac{|w_{k+1}(x)-w_{k+1}(y)|^{p}}{|x-y|^{n+ps_1}}dxdy \le A_p(v,w_{k+1}) =(\rho\|u\|_{L^r(\Om)})^{1-p}A_p(u, w_{k+1}) \]
	   and \[ 0\le \|w_{k+1}\|_{X_{q,s_2}}^{q}\le A_q(v,w_{k+1}) =(\rho\|u\|_{L^r(\Om)})^{1-q} A_q(u, w_{k+1}).\]
  Set $\Om_k =\Om\cap\{x\in\Om :w_k(x)>0\}$. Using the fact $u$ is a weak solution of $(\mc P_\la)$, we deduce that
	  \begin{align}\label{eqb33}
	    \|w_{k+1}\|_{X_{p,s_1}}^{p}&\le (\rho\|u\|_{L^r(\Om)})^{1-p}A_p(u, w_{k+1})+ (\rho\|u\|_{L^r(\Om)})^{1-p}A_q(u, w_{k+1}) \nonumber \\
	    &=
	    (\rho\|u\|_{L^r(\Om)})^{1-p}\left(\ds\int_{\Om_{k+1}}(\la a(x)u^{\de-1}+b(x)u^{r-1})w_{k+1}(x)dx \right) \nonumber\\
        &=(\rho\|u\|_{L^r(\Om)})^{1-p} \left(\ds\int_{\Om_{k+1}\cap\{u(x)\le 1\}}+ \ds\int_{\Om_{k+1}\cap\{u(x)>1\}} \right)(\la a(x)u^{\de-1}+b(x)u^{r-1})w_{k+1}dx  \nonumber\\
    	&\le C_\la(\rho\|u\|_{L^r(\Om)})^{1-p} \left( \ds\int_{\Om_{k+1}\cap\{u(x)\le 1\}} u^{\de-1}(x)w_{k+1}dx \right.\nonumber\\
    	&\qquad \qquad \qquad\qquad\qquad \left.
    	 +\ds\int_{\Om_{k+1}\cap\{u(x)>1\}} u^{r-1}(x)w_{k+1}dx \right).
	 \end{align}
   Noting that $(\rho\|u\|_{L^r(\Om)})^{-1}\le 1$, we have $v(x)-1+2^{-(k+1)}\le 2^{-(k+1)}$ whenever $u(x)\le 1$, that is $w_{k+1}(x)\leq 2^{-(k+1)}$.
   Also, $u(x)=(\rho\|u\|_{L^r(\Om)})v(x)< (\rho\|u\|_{L^r(\Om)})(2^{k+1}-1)w_k$ for all $x\in\Om_{k+1}$.
   Thus, \eqref{eqb33} reduces to
	 \begin{equation}\label{eqb35}
	 \begin{aligned}
	   \|w_{k+1}\|_{X_{p,s_1}}^{p}&\le C_\la (\rho\|u\|_{L^r(\Om)})^{1-p}\left(2^{-(k+1)}|\Om_{k+1}|+ 
	   (\rho\|u\|_{L^r(\Om)})^{r-1}(2^{k+1}-1)^{r-1}\ds\int_{\Om_{k+1}} w_k(x)^rdx \right) \\
	   &\le  C_\la (\rho\|u\|_{L^r(\Om)})^{1-p}\left(2^{(r-1)(k+1)}U_k+ 2^{(k+1)(r-1)}(\rho\|u\|_{L^r(\Om)})^{r-1}U_k\right) \\ &\le 2\;C_\la\; 2^{(r-1)(k+1)} (\rho\|u\|_{L^r(\Om)})^{r-p}\;U_k.
	 \end{aligned}
	\end{equation}
  With the help of \eqref{eqb35} and Sobolev embedding, we have
	 \begin{equation}\label{eqb36}
	  \begin{aligned}
	    U_{k+1}=\|w_{k+1}\|_{L^r(\Om)}^r&\le \left(\ds\int_\Om w_{k+1}^{p^*_{s_1}}\right)^\frac{r}{p^*_{s_1}}|\Om_{k+1}|^{1-\frac{r}{p^*_{s_1}}} \\
	   &\le \left(S^{-1}\|w_{k+1}\|_{X_{p,s_1}}^{p} \right)^\frac{r}{p} |\Om_{k+1}|^{1-\frac{r}{p^*_{s_1}}} \\
	  &\le S^\frac{-r}{p} \left(2C_\la 2^{(r-1)(k+1)} (\rho\|u\|_{L^r(\Om)})^{r-p}U_k  \right)^\frac{r}{p}\left(2^{(k+1)r}U_k\right)^{1-\frac{r}{p^*_{s_1}}} \\
	  &\leq C^k (\rho\|u\|_{L^r(\Om)})^{\frac{r^2}{p}-r} U_k^{1+\frac{rs_1}{n}},
	\end{aligned}
	\end{equation}
  where $C>1$ independent of $k$.
  Let $\eta= C^{-\frac{n}{rs_1}}\in (0,1)$ and define $\ga := \frac{r^2s_1}{n}+r- \frac{r}{p} >0$. Choose $\rho$ such that 
   $$ \rho\ge \max\bigg\{1,\|u\|_{L^r(\Om)}^{-1}, \big((\|u\|_{L^r(\Om)})^{\frac{r^2}{p}-r}\eta^{-1}\big)^\frac{1}{\ga}, C^\frac{n^2}{(rs_1)^2}\bigg\}.$$
   We claim that $U_k\le\frac{\eta^k}{\rho^r}$, which we prove by induction.
	For $k=0$, we have $U_0=\|v^+\|_{L^r(\Om)}^r\leq \|v\|_{L^r(\Om)}^r=\rho^{-r}$. 
	Let us assume $U_k\le\frac{\eta^k}{\rho^r}$ for some $k\in\mb N$, then using  \eqref{eqb36}, we have
	   \begin{align*}
	    U_{k+1}&\leq C^k (\rho\|u\|_{L^r(\Om)})^{\frac{r^2}{p}-r} U_k^{1+\frac{rs_1}{n}}\\ &\leq C^k (\rho\|u\|_{L^r(\Om)})^{\frac{r^2}{p}-r}\left(\frac{\eta^k}{\rho^r}\right)^{1+\frac{rs_1}{n}}\\&
     	\leq \frac{\eta^k}{\rho^r} (\|u\|_{L^r(\Om)})^{\frac{r^2}{p}-r} \rho^{-\ga}
	    \leq \frac{\eta^{k+1}}{\rho^r}.
	   \end{align*}
  Hence the claim follows.\\
  Now by the claim we get $U_k\ra 0$ as $k\ra\infty$, which gives us $w_k(x)\ra 0$ \text{a.e.} in $\Om$, hence $v(x)\le 1$ a.e. in $\Om$ and similar arguments yield $-v(x)\le 1$ a.e. in $\Om$. Thus $\|v\|_{L^\infty(\Om)}\le 1$, which implies $\|u\|_{L^\infty(\Om)}\le\rho\|u\|_{L^r(\Om)}$, that is $u\in L^\infty(\Om)$.\QED
  \end{proof}
    \begin{Definition}
    	Let $\Om\subset\mb R^n$ be a bounded set. Define
	      \begin{align*}
	        \widetilde{W}^{s,p}(\Om):= \bigg\{u\in L^p_{\rm loc}(\mb R^n): \exists\;U\Supset\Om \ s.t.\ \|u\|_{W^{s,p}(U)}+\int_{\mb R^n} \frac{|u(x)|^{p-1}}{(1+|x|)^{n+sp}}dx <\infty \bigg\}.
	  \end{align*}\end{Definition}
     \begin{Lemma}\label{decom}
  	 Let $\Om\subset\mb R^n$ be a bounded domain, and $u\in\widetilde{W}^{s,p}(\Om)$, $v\in L^1_{loc}(\mb R^n)$ are such that 
  		\begin{align}\label{eqbv}
  	     	dist(K, \Om)>0,\quad \int_{\Om^c}\frac{|v(x)|^{p-1}}{(1+|x|)^{n+ps}}dx<\infty,
  		\end{align}
  	 where $K=\text{supp(v)}$. Define 
  		{\small\begin{align*}
  			h(x)=2 \int_{K} \frac{|u(x)-u(y)-v(y)|^{p-2}(u(x)-u(y)-v(y))-|u(x)-u(y)|^{p-2}(u(x)-u(y))}{|x-y|^{n+ps}}dy.
  			\end{align*}\small}
  	  Then, $u+v\in\widetilde{W}^{s,p}(\Om)$ and $(-\Delta)^s_p(u+v)=(-\Delta)_p^s
  		u +h$ weakly in $\Om$. \end{Lemma}
  	\begin{proof} 
  	 Since $u\in\widetilde{W}^{s,p}(\Om)$, there exists an open set $U\Supset\Om$ such that 
  		\begin{equation} \label{eqbu}
  		\|u\|_{W^{s,p}(U)}+\ds\int_{\mb R^n} \frac{|u(x)|^{p-1}}{(1+|x|)^{n+sp}}dx <\infty.\end{equation} 
  	 Without loss of generality assume $\Om\Subset U\Subset K^c$, then it is easy to observe that $u+v=u$ in $U$, hence $u+v\in W^{s,p}(U)$. Moreover, using \eqref{eqbu} and \eqref{eqbv}, we deduce that
  		\begin{align*}
  		 \int_{\mb R^n} \frac{|u(x)+v(x)|^{p-1}}{(1+|x|)^{n+sp}}dx \le C\left(\int_{\mb R^n} \frac{|u(x)|^{p-1}}{(1+|x|)^{n+sp}}dx+\int_{K} \frac{|v(x)|^{p-1}}{(1+|x|)^{n+sp}}dx   \right)<\infty.
  		\end{align*}
  	 For any $\phi\in C_c^\infty(\Om)$, we have
  		\begin{align*}
  		   \int_{\mb R^n}\int_{\mb R^n}& \frac{|u(x)+v(x)-u(y)-v(y)|^{p-2}(u(x)+v(x)-u(y)-v(y))(\phi(x)-\phi(y))}{|x-y|^{n+sp}}dxdy \\
  		   =& \int_{\Om}\int_{\Om} \frac{|u(x)-u(y)|^{p-2}(u(x)-u(y))(\phi(x)-\phi(y))}{|x-y|^{n+sp}}dxdy  \\
  		   &+ \int_{\Om}\int_{\Om^c} \frac{|u(x)-u(y)-v(y)|^{p-2}(u(x)-u(y)-v(y))\phi(x)}{|x-y|^{n+sp}} dxdy \\
  		   &- \int_{\Om^c}\int_{\Om} \frac{|u(x)+v(x)-u(y)|^{p-2}(u(x)+v(x)-u(y))\phi(y)}{|x-y|^{n+sp}} dxdy \\
  		   =&  \int_{\mb R^n}\int_{\mb R^n} \frac{|u(x)-u(y)|^{p-2}(u(x)-u(y))(\phi(x)-\phi(y))}{|x-y|^{n+sp}}dxdy \\
  		  &- 2  \int_{\Om}\int_{\Om^c} \frac{|u(x)-u(y)|^{p-2}(u(x)-u(y))\phi(x)}{|x-y|^{n+sp}} dxdy \\
  		  &+ 2 \int_{\Om}\int_{\Om^c} \frac{|u(x)-u(y)-v(y)|^{p-2}(u(x)-u(y)-v(y))\phi(x)}{|x-y|^{n+sp}} dxdy. 
  		\end{align*}
  	Since $C_c^\infty(\Om)$ is dense in $W^{s,p}_0(\Om)$, we obtain
  		\begin{align*}
  		(-\Delta)_p^s(u+v)= (-\Delta)_p^s u +h \quad \text{weakly in } \Om.
  		\end{align*}\QED
  \end{proof}       
  Let $\langle ._,.\rangle_{s,p,\Om}$ denotes the duality pair between $W^{s,p}_0(\Om)$ and $W^{-s,p'}(\Om)$. For $f\in W^{-s,p'}(\Om)$, we say $(-\Delta)_p^s u\ge f$ weakly in $\Om$ if
  \begin{align*}
  \int_{\mb R^n}\int_{\mb R^n} \frac{|u(x)-u(y)|^{p-2}(u(x)-u(y))(\phi(x)-\phi(y)}{|x-y|^{n+sp}}dxdy\ge\langle f,\phi\rangle_{s,p,\Om}
  \end{align*}
  holds for all $\phi\in W^{s,p}_0(\Om)$ with $\phi\ge 0$ a.e. in $\Om$.
  
  \begin{Proposition}(Comparison principle).\label{comp} 
   Let $\Om\subset\mb R^n$ be a bounded domain, and  $u,v\in\widetilde{W}^{s_1,p}(\Om)\cap\widetilde{W}^{s_2,q}(\Om)$ are such that $u\le v$ in $\Om^c$. Suppose \begin{align*}
		\int_{\mb R^n}\int_{\mb R^n} &\frac{|u(x)-u(y)|^{p-2}(u(x)-u(y))(\phi(x)-\phi(y))}  {|x-y|^{n+p s_1}}~ dxdy\\&\quad +\ba\int_{\mb R^n}\int_{\mb R^n}\frac{|u(x)-u(y)|^{q-2}(u(x)-u(y))(\phi(x)-\phi(y))}  {|x-y|^{n+q s_2}}~ dxdy \\
		\le & \int_{\mb R^n}\int_{\mb R^n} \frac{|v(x)-v(y)|^{p-2}(v(x)-v(y))(\phi(x)-\phi(y))} {|x-y|^{n+p s_1}}~ dxdy\\ &\quad +\ba\int_{\mb R^n}\int_{\mb R^n} \frac{|v(x)-v(y)|^{q-2}(v(x)-v(y))(\phi(x)-\phi(y))} {|x-y|^{n+q s_2}}~ dxdy
		\end{align*}
  whenever $\phi\in W^{s_1,p}_0(\Om)\cap W^{s_2,q}_0(\Om)$ and $\phi\ge 0$ in $\Om$. Then, $u\le v$ in $\Om$.\end{Proposition}
  \begin{proof}
	The proof follows in the same line of \cite[Proposition 2.10]{iann}.\QED
   \end{proof}

\begin{Lemma}\label{harnc}(Weak Harnack Inequality). Let $2\le q\le p<\infty$, then
	there exist $\sg\in(0,1)$ and $\bar{C}_1>0$ such that if $u\in\widetilde{W}^{s_1,p}(B_{R/3})\cap\widetilde{W}^{s_2,q}(B_{R/3})$ satisfies weakly
	\begin{align*}
	(-\Delta)_p^{s_1} u +\ba(-\Delta)_q^{s_2} u &\ge -K \quad \mbox{in } B_{R/3}, \\
	u\ge 0 \quad \mbox{in } \mb R^n,
	\end{align*}
	for some $K>0$, then 
	\begin{align*}
	\inf_{B_{R/4}} u \ge \sg\left(\Xint-_{B_{R}\setminus B_{R/2}} u^{q-1}\right)^{\frac{1}{q-1}} -\bar{C}_1(KR^{ps_1})^{\frac{1}{p-1}}.
	\end{align*}\end{Lemma}	
\begin{proof}
	Let $\phi\in C_c^{\infty}(\mb R^n)$ be such that $0\le\phi\le 1$ in $\mb R^n$, $\phi =1$ in $B_{3/4}$, $\phi=0$ in $B_1^c$. Using \cite[Proposition 2.12]{iann}, we have
	$|(-\Delta)_p^{s_1} \phi|\leq C_1$ and $|(-\Delta)_q^{s_2} \phi|\leq C_1$ weakly in $B_1$, for some $C_1>0$.
	Define $\phi_R(x)=\phi(3x/R)$, then $\phi_R\in C_c^{\infty}(\mb R^n)$, $0\le\phi_R\le 1$ in $\mb R^n$, $\phi_R =1$ in $B_{R/4}$ and $\phi=0$ in $B_{R/3}^c$. Moreover, $|(-\Delta)_p^{s_1} \phi_R|\leq C_1 R^{-ps_1}$ and $|(-\Delta)_q^{s_2} \phi_R|\leq C_1 R^{-qs_2}$ weakly in $B_{R/3}$.
	For $\sg\in(0,1)$, set 
	\begin{align*}
	L= \left(\Xint-_{B_R\setminus B_{R/2}} u^{q-1}\right)^{\frac{1}{q-1}} \mbox{and} \quad w=\sg L\phi_R +\chi_{B_R\setminus B_{R/2}}u.	
	\end{align*}
	Using Lemma \ref{decom}, we have $w\in\widetilde{W}^{s_1,p}(B_{R/3})\cap\widetilde{W}^{s_2,q}(B_{R/3})$ with
	\begin{align*}
	(-\Delta)_p^{s_1}w +\ba (-\Delta)_q^{s_2}w=  \big((-\Delta)_p^{s_1}+\ba(-\Delta)_q^{s_2}\big) (\sg L\phi_R) +h_p+h_q \quad \text{weakly in } B_{R/3},
	\end{align*}
 where \begin{align*}
	h_p(x)= 2\ds\int_{B_R\setminus B_{R/2}} \frac{|\sg L\phi_R(x)-u(y)|^{p-2}(\sg L\phi_R(x)-u(y))-(\sg L\phi_R(x))^{p-1}}{|x-y|^{n+ps_1}}~dy \ \ \mbox{and} \\
	h_q(x)= 2\ds\int_{B_R\setminus B_{R/2}} \frac{|\sg L\phi_R(x)-u(y)|^{q-2}(\sg L\phi_R(x)-u(y))-(\sg L\phi_R(x))^{q-1}}{|x-y|^{n+qs_2}}~dy.\end{align*}
	Now using the inequality: for $a,b\ge 0$ and $\ga\ge 1$, 
	\begin{equation*} 
	a^\ga -|a-b|^{\ga}(a-b)\ge 2^{1-\ga}\;b^{\ga},
	\end{equation*}
	it follows that
	\begin{align*}
	&h_p(x)\leq -2^{3-p}\int_{B_R\setminus B_{R/2}}\frac{(u(y))^{p-1}}{|x-y|^{n+ps_1}}dy \ \text{ and } \ h_q(x) \leq -2^{3-q}\int_{B_R\setminus B_{R/2}} \frac{(u(y))^{q-1}}{|x-y|^{n+qs_2}}dy. 
	\end{align*}
	Hence, for a.e. $x\in B_{R/3}$
	\begin{align*}
	  \big((-\Delta)_p^{s_1}+\ba (-\Delta)_q^{s_2}\big)w(x) &\leq \frac{C_1(\sg L)^{p-1}}{R^{ps_1}}+\frac{C_1(\sg L)^{q-1}}{R^{qs_2}} -2^{3-p}\int_{B_R\setminus B_{R/2}}\frac{(u(y))^{p-1}}{|x-y|^{n+ps_1}}~dy \\ &\qquad\qquad -2^{3-q}\int_{B_R\setminus B_{R/2}} \frac{(u(y))^{q-1}}{|x-y|^{n+qs_2}}~dy\\
	  &\leq\frac{C_1(\sg L)^{p-1}}{R^{ps_1}}+\frac{C_1(\sg L)^{q-1}}{R^{qs_2}}-\frac{C_2\; L^{p-1}}{R^{ps_1}}-\frac{C_3\; L^{q-1}}{R^{qs_2}}\\&
	  \leq -\frac{C_2\; L^{p-1}}{2R^{ps_1}}-\frac{C_3\; L^{q-1}}{2R^{qs_2}},
	\end{align*}
	where the last inequality follows if we assume 
	\begin{align*}
	\sg < \min\Big\{1,\left(\frac{C_2}{2C_1}\right)^\frac{1}{p-1},  \left(\frac{C_3}{2C_1}\right)^\frac{1}{q-1}\Big\}.
	\end{align*}
	Therefore, we have
	\begin{align*}
	\mc L w(x)\leq -\frac{C_2\; L^{p-1}}{2R^{ps_1}} \quad\text{ weakly in } B_{R/3}.
	\end{align*}
	For convenience denote $\bar{C}_1 =\Big(\frac{1}{C_2}\Big)^\frac{1}{p-1}$, then we consider following cases:\\
	\textbf{Case 1:} If $L\le\bar{C}_1(KR^{ps_1})^\frac{1}{p-1}$.\\
	By using the fact that $0<\sg<1$, we have $\sg L\le \bar{C}_1(KR^{ps_1})^\frac{1}{p-1}$. Hence, $$ \inf_{B_{R/4}}u\ge 0\ge \sg L- \bar{C}_1(KR^{ps_1})^\frac{1}{p-1}.$$
	\textbf{Case 2:} If $L>\bar{C}_1(KR^{ps_1})^\frac{1}{p-1}$. \\
	In this case $L^{p-1}\ge\frac{KR^{ps_1}}{C_2}$, which gives us 
	\begin{align*}
	&\mc L w\le -K\le\mc L u\qquad\text{weakly in } B_{R/3}\\
	&w=\chi_{B_R\setminus B_{R/2}} u\le u \quad \text{in } B^c_{R/3}.
	\end{align*}
	Then by proposition \ref{comp}, we deduce that $w\le u$ in $\mb R^n$, in particular 
	$$ \inf_{B_{R/4}} u\ge\inf_{B_{R/4}} w =\sg L.$$
	Thus, 
	\begin{align*}
	\inf_{B_{R/4}} u \ge \sg\left(\Xint-_{B_{R}\setminus B_{R/2}} u^{q-1}\right)^{\frac{1}{q-1}} -\bar{C}_1(KR^{ps_1})^{\frac{1}{p-1}}.
	\end{align*}\QED
\end{proof}
\begin{Definition}
	For $u:\mb R^n\to \mb R$ measurable, $1<p,q<\infty$ and $s_1,s_2\in(0,1)$, the nonlocal tail centered at $x\in\mb R^n$ with radius $R>0$ is defined as 
	\begin{align*}
	T_p(u;x;R)=\left(R^{ps_1}\int_{B_R^c} \frac{|u(y)|^{p-1}}{|x-y|^{n+ps_1}} dy\right)^\frac{1}{p-1},
	\end{align*}
	and analogously $T_q(u;x;R)$ is defined. Set $T_p(u;R)=T_p(u;0;R)$ and $T_q(u;R)=T_q(u;0;R)$.
\end{Definition}

\begin{Lemma}\label{imharnc}
	Let $2\le q\le p$, then there exist $\sg>0$, $\bar{C}>0$, and for all $\e>0$, a constant $C_\e>0$ such that if  $u\in\widetilde{W}^{s_1,p}(B_{R/3})\cap\widetilde{W}^{s_2,q}(B_{R/3})$ is bounded in $\mb R^n$ and satisfies weakly
	\begin{align*}
	(-\Delta)_p^{s_1} u +\ba(-\Delta)_q^{s_2} u &\ge -K \quad \mbox{in } B_{R/3}, \\
	u\ge 0 \quad \mbox{in } B_R,
	\end{align*}
	for some $K>0$, then there exists $M>0$ such that
	{\small\begin{align*}
	\inf_{B_{R/4}} u \ge \sg\left(\Xint-_{B_{R}\setminus B_{R/2}} u^{q-1}\right)^{\frac{1}{q-1}} -\bar{C}(KR^{ps_1})^{\frac{1}{p-1}}- \e\ds\sup_{B_R} u-C_\e T_p(u_-;R)-R^\frac{ps_1-qs_2}{p-1}M(\e+C_\e).
	\end{align*}\small}
\end{Lemma}	
\begin{proof}
	We apply Lemma \ref{decom} to $u$ and $v=u_-$ for $\Om=B_{R/3}$, then we have in weak sense in $B_{R/3}$
	\begin{align*}
	(-\De)_p^{s_1}u_+(x)= (-\De)_p^{s_1}u(x)+2\;g_p(x),  \end{align*}
	where {\small\begin{align*}
	g_p(x)=&\ds \int_{B_{R/3}^c}  \frac{|u(x)-u(y)-u_-(y)|^{p-2}(u(x)-u(y)-u_-(y))-|u(x)-u(y)|^{p-2}(u(x)-u(y))}{|x-y|^{n+ps_1}}dy\\
		=& \int_{\{u<0\}}  \frac{|u(x)-u(y)-u_-(y)|^{p-2}(u(x)-u(y)-u_-(y))-|u(x)-u(y)|^{p-2}(u(x)-u(y))}{|x-y|^{n+ps_1}}dy.
		\end{align*}\small}
	Since $u\ge 0$ in $B_R$ and $u+u_-=u_+$, function $g_p$ reduces to 
	\begin{align*}
	g_p(x)
	=& \int_{\{u<0\}} \frac{(u(x))^{p-1}-|u(x)-u(y)|^{p-2}(u(x)-u(y))}{|x-y|^{n+ps_1}}dy
	\end{align*}
	a.e. in $B_{R/3}$. Using the fact that $\{u<0\}\subset B_R^c$, we have $|x-y|\ge \frac{2}{3}|y|$ for all $x\in B_{R/3}$ and $y\in\{ u<0\}$. Thus, a.e. $x\in B_{R/3}$ 
	\begin{align*}
	g_p(x)\ge -C\int_{\{u<0\}} \frac{|u(x)-u(y)|^{p-2}(u(x)-u(y))-u^{p-1}(x)}{|y|^{n+ps_1}}dy.
	\end{align*}
	Now using the inequality: for all $\theta >0$, there exists a constant $C_\theta>0$ such that for all $a,b\ge 0$ and $m\ge 1$, 
	\[(a+b)^m -a^m \le \theta \;a^m+ C_\theta \;b^m, \]
	we deduce that 
	\begin{align*}
	g_p(x)&\ge -C\theta u^{p-1}(x)\int_{B_R^c} \frac{1}{|y|^{n+ps_1}}dy -C_\theta \int_{B_R^c} \frac{|u(y)|^{p-1}}{|y|^{n+ps_1}}dy\\
&	\ge \frac{-C\theta}{R^{ps_1}} \Big(\sup_{B_R} u\Big)^{p-1} -\frac{C_\theta}{R^{ps_1}}(T_p(u_-;R))^{p-1}
	\end{align*} 
	a.e. in $ B_{R/3}$. Set 
	{\small\begin{align*}
	g_q(x)=\ds \int_{B_{R/3}^c}  \frac{|u(x)-u(y)-u_-(y)|^{q-2}(u(x)-u(y)-u_-(y))-|u(x)-u(y)|^{q-2}(u(x)-u(y))}{|x-y|^{n+qs_2}}dy,
	\end{align*}\small}
	then analogous argument shows that 
	\[(-\De)_q^{s_2}u_+(x)= (-\De)_q^{s_2}u(x)+2\;g_q(x),\quad \mbox{ weakly in }B_{R/3} \]
	and 
	\begin{align*}
	2\;g_q(x)  &\ge \frac{-C\theta}{R^{qs_2}} \Big(\sup_{B_R} u\Big)^{q-1} -\frac{C_\theta}{R^{qs_2}}(T_q(u_-;R))^{q-1}
	\end{align*}
	a.e. in $B_{R/3}$. Thus, collecting all the results, we have 
	\begin{align*}
	\big((-\De)_p^{s_1}+\ba (-\De)_q^{s_2}\big)\; u_+(x)&\ge (-\De)_p^{s_1}u(x)+\ba (-\De)_q^{s_2}u(x)
	-\frac{C\theta}{R^{ps_1}} \Big(\sup_{B_R} u\Big)^{p-1} \\ & \;-\frac{C_\theta}{R^{ps_1}}(T_p(u_-;R))^{p-1} 
	-\frac{C\theta}{R^{qs_2}} \Big(\sup_{B_R} u\Big)^{q-1} -\frac{C_\theta}{R^{qs_2}}(T_q(u_-;R))^{q-1} \\
	&\ge -K -\frac{C\theta}{R^{ps_1}} \Big(\sup_{B_R} u\Big)^{p-1}  -\frac{C_\theta}{R^{ps_1}}(T_p(u_-;R))^{p-1} \\ &\quad
	-\frac{C\theta}{R^{qs_2}} \Big(\sup_{B_R} u\Big)^{q-1}  -\frac{C_\theta}{R^{qs_2}}(T_q(u_-;R))^{q-1}:= -\hat{K}. \end{align*}
	Since $u$ is bounded in $\mb R^n$, we have $\ds\sup_{B_R} u\le \|u\|_{L^\infty(\mb R^n)}$ and 
	\begin{align*}
	T_q(u_-;R)^{q-1}=R^{qs_2}\int_{B_R^c}\frac{|u(y)|^{q-1}}{|y|^{n+qs_2}}dy\le C \frac{\|u\|_{L^\infty(\mb R^n)}^{q-1}}{qs_2},
	\end{align*}
	where $C>0$ is a constant independent of $R$. Hence, there exists a constant $M>0$ independent of $R$ such that \[\Big(\ds\sup_{B_R} u\Big)^\frac{q-1}{p-1}\le M \ \ \mbox{and} \ \ \big(T_q(u_-;R)\big)^\frac{q-1}{p-1}\le M.\] 
	Using the inequality $a^m+b^m\le 2^{1-m}(a+b)^m$ for $a,b\ge 0$, and $0<m\le 1$, we deduce that 
	\begin{align*}
	\bar{C_1}(\hat{K}R^{ps_1})^\frac{1}{p-1}\le& \bar{C}(KR^{ps_1})^\frac{1}{p-1}+\bar{C}(C\theta)^\frac{1}{p-1}\Big(\sup_{B_R} u\Big)+ \bar{C}C_\theta^\frac{1}{p-1}T_p(u_-;R)\\
	&+\bar{C}R^\frac{ps_1-qs_2}{p-1}M\Big((C\theta)^\frac{1}{p-1}+C_\theta^\frac{1}{p-1} \Big),
	\end{align*}
	where $\bar{C}>0$ is a constant. For all $\e>0$, choose $\theta>0$ such that $C\theta <\Big(\frac{\e}{\bar{C}}\Big)^{p-1}$, then there exists a constant $C_\e>0$ such that
	\begin{align*}
	\bar{C_1}(\hat{K}R^{ps_1})^\frac{1}{p-1}\le \bar{C}(KR^{ps_1})^\frac{1}{p-1}+\e\sup_{B_R} u+C_\e \;T_p(u_-;R)
	+R^\frac{ps_1-qs_2}{p-1}M \big(\e+C_\e \big).
	\end{align*}
	Applying Lemma \ref{harnc} to $u_+$, for any $\e>0$, we have 
	\begin{align*}
	\inf_{B_{R/4}} u=\inf_{B_{R/4}} u_+ 
	\ge &\sg\left(\Xint-_{B_{R}\setminus B_{R/2}} u^{q-1}\right)^{\frac{1}{q-1}}-\bar{C}(KR^{ps_1})^\frac{1}{p-1}-\e\sup_{B_R} u-C_\e \;T_p(u_-;R)\\
	&\quad-R^\frac{ps_1-qs_2}{p-1} M\big(\e+C_\e \big).
	\end{align*} \QED
\end{proof} 
\begin{Theorem}\label{oscl} Let $2\le q\le p$.
	There exist $\al\in(0,1)$ and $C>0$ such that if $u\in\widetilde{W}^{s_1,p}(B_{R_0})\cap\widetilde{W}^{s_2,q}(B_{R_0})$ is bounded in $\mb R^n$ and satisfies $|(-\Delta)_p^{s_1}u+\ba(-\Delta)_q^{s_2}|\le K$ weakly in $B_{R_0}$ for some $K$ and $R_0>0$, then for all $r\in(0, R_0)$ 
	\[ \underset{B_r}{\rm osc}\; u \le C\Big((KR_0^{ps_1})^\frac{1}{p-1}+Q(u;R_0)+R_0^\frac{ps_1-qs_2}{p-1} \Big)\frac{r^\al}{R_0^\al}.\]
\end{Theorem}
\begin{proof}
	For all $j\in\mb N\cup\{0\}$, define sequences $R_j=\frac{R_0}{4^j}$, $B_j=B_{R_j}$, $\frac{1}{2}B_j=B_{R_j/2}$ and $A_j= B_j\setminus\frac{1}{2}B_j$.
	We claim that there exist $\al>0$, $\mu>0$, a non-decreasing sequence $\{m_j\}$ and a non-increasing sequence $\{M_j\}$ such that  
	\begin{align*}
	m_j\le\inf_{B_j}\; u\le\sup_{B_j}\; u\le M_j ,\qquad M_j-m_j=\mu R_j^\al.
	\end{align*}
	We will proceed by induction. For $j=0$, we set $M_0= \|u\|_{L^\infty(B_{R_0})}$ and $m_0= M_0-\mu R_0^\al$, where $\mu$ satisfies 
	\begin{equation}\label{eqmu}
	\mu\ge \frac{2\|u\|_{L^\infty(B_{R_0})}}{R_0^\al}>0.
	\end{equation}
	Hence, $m_0\le\ds\inf_{B_0}\; u\le\ds\sup_{B_0}\; u\le M_0.$ Suppose the claim holds for all $i\in\{0,\dots, j\}$ for some $j\in\mb N\cup\{0\}$. Then,
	\begin{align*}
	M_j-m_j&= \Xint-_{A_j}(M_j-u(x))dx+\Xint-_{A_j}(u(x)-m_j)dx \\
	&\le \left(\Xint-_{A_j}(M_j-u)^{q-1}\right)^\frac{1}{q-1}+ \left(\Xint-_{A_j}(u-m_j)^{q-1}\right)^\frac{1}{q-1}.
	\end{align*}
	Now employing lemma \ref{imharnc}, we obtain 
	\begin{align*}
	\sg(M_j-m_j)&\le\inf_{B_j}(M_j-u)+\inf_{B_j}(u-m_j)+2\; \bar{C}(KR_j^{ps_1})^\frac{1}{p-1}+\e\Big(\sup_{B_j}(M_j-u)+\sup_{B_j}(u-m_j)\Big) \\ &\quad+C_\e\big(T_p((M_j-u)_-;R_j)+T_p((u-m_j)_-;R_j)\big)+M(C_\e+\e)R_j^\frac{ps_1-qs_2}{p-1},
	\end{align*} 
	where $\sg\in (0,1)$ and $\bar{C},\; {C}_\e >0$ are as in lemma \ref{imharnc}. Set $\e=\sg/4$ and $C=\max\{2\bar{C}, M(C_\e+\e), C_\e\}$. Thus, noting that $\ds\sup_{B_j}(M_j-u)+\ds\sup_{B_j}(u-m_j)\le 2(M_j-m_j)$, we have
	 \begin{equation}\label{eqb60}
		\begin{aligned}
		\underset{B_{j+1}}{\rm osc}\; u\le& \left(1-\frac{\sg}{2}\right)(M_j-m_j)\\ &+ C\left((KR_j^{ps_1})^\frac{1}{p-1}+T_p((M_j-u)_-;R_j)+T_p((u-m_j)_-;R_j)+R_j^\frac{ps_1-qs_2}{p-1} \right).
		\end{aligned}
		\end{equation}
	Following the proof of \cite[Theorem 5.4]{iann}, we have estimate on the nonlocal tails
	\begin{align*}
	T_p((u-m_j)_-;R_j)\le C\left( \mu S(\al)^\frac{1}{p-1}+ \frac{Q(u;R_0)}{R_0^\al}\right)R_j^\al, 
	\end{align*}
	where $S(\al)= \ds\sum_{k=1}^\infty \frac{(4^{\al k}-1)^{p-1}}{4^{ps_1k}}\ra 0$ as $\al\ra0^+$, $Q(u;R_0)=\|u\|_{L^\infty(\mb R^n)}+T_p(u;R_0)$ and $\al <\frac{ps_1}{p-1}.$ Analogous estimate holds for $T_p((M_j-u)_-;R_j)$. Thus, for $\al<\frac{ps_1-qs_2}{p-1}$, \eqref{eqb60} implies
	\begin{equation}
	\begin{aligned}
	\underset{B_{j+1}}{\rm osc}\; u&\le \left(1-\frac{\sg}{2}\right)(M_j-m_j)+ C\left((KR_j^{ps_1})^\frac{1}{p-1}+\Big(\mu S(\al)^\frac{1}{p-1}+ \frac{Q(u;R_0)}{R_0^\al}\Big)R_j^\al+R_j^\frac{ps_1-qs_2}{p-1} \right) \\
	&\le \left(1-\frac{\sg}{2}\right)\mu R_j^\al+ C\left(K^\frac{1}{p-1}R_0^{\frac{ps_1}{p-1}-\al} +\mu S(\al)^\frac{1}{p-1}+ \frac{Q(u;R_0)}{R_0^\al}+R_0^{\frac{ps_1-qs_2}{p-1}-\al }\right)R_j^\al \\
	&= 4^\al\left(\Big(1-\frac{\sg}{2}\Big)+CS(\al)^\frac{1}{p-1} \right)\mu R_{j+1}^\al +\frac{4^\al C}{R_0^\al}\Big((KR_0^{ps_1})^\frac{1}{p-1}+ Q(u;R_0)+R_0^\frac{ps_1-qs_2}{p-1} \Big)R_{j+1}^\al \label{eqb58}
	\end{aligned}
	\end{equation}
	Choose $\al<\frac{ps_1-qs_2}{p-1}$ such that \[4^\al\left(\Big(1-\frac{\sg}{2}\Big)+CS(\al)^\frac{1}{p-1} \right)\le 1-\frac{\sg}{4}\] and set
	\begin{equation}\label{eqb59}
	\mu=\frac{4^{\al+1}\; C}{\sg\;R_0^\al}\Big((KR_0^{ps_1})^\frac{1}{p-1}+ Q(u;R_0)+R_0^\frac{ps_1-qs_2}{p-1} \Big).
	\end{equation} 
	If $\mu<\frac{2\|u\|_{L^\infty(B_{R_0})}}{R_0^\al}$, then we can replace $C$ appearing in \eqref{eqb60} by bigger constant so that \eqref{eqmu} hold true. Thus, from \eqref{eqb58}, we have
	\begin{align*}
	\underset{B_{j+1}}{\rm osc}\; u\le \mu R_{j+1}^\al.
	\end{align*}
	Therefore, we pick $m_{j+1},\; M_{j+1}$ such that 
	\begin{align*}
	m_j\le m_{j+1}\le\inf_{B_j}\; u\le\sup_{B_j}\; u\le M_{j+1}\le M_j \quad \mbox{and } M_{j+1}-m_{j+1}= \mu R_{j+1}^\al.
	\end{align*}
	Fix $r\in(0,R_0)$. Let $j\in\mb N\cup\{0\}$ be such that $R_{j+1}\le r\le R_j$, then taking into account \eqref{eqb59} and $R_j\le 4r$, we have 
	\begin{align*}
	\underset{B_r}{\rm osc}\; u\le\underset{B_j}{\rm osc} \;u\le\mu R_j^\al \le C\Big((KR_0^{ps_1})^\frac{1}{p-1}+Q(u;R_0)+R_0^\frac{ps_1-qs_2}{p-1} \Big)\frac{r^\al}{R_0^\al}.
	\end{align*}\QED
\end{proof}  

 {\bf Proof of Theorem \ref{thm1.1}} Proof of $L^\infty$ bound follows from Theorem \ref{linfty}. Next we prove interior H\"older regularity for weak solutions of $(\mc P_\la)$ when $2\le q\le p<r<  p_{s_1}^*$. Let $u$ be a nontrivial weak solution of $(\mc P_\la)$. Assume the function $a\in L^\infty(\Om)$ and let $R_0>0$ be such that $B_{R_0}:=B_{R_0}(0)\Subset\Om$. Since $u\in L^\infty(\Om)\cap X_{p,s_1}$, therefore $u\in\widetilde{W}^{s_1,p}(B_{R_0})\cap \widetilde{W}^{s_2,q}(B_{R_0})$ and is bounded in $\mb R^n$. It implies that
  \[|(-\De)_{p}^{s_1 }u+\ba(-\De)_q^{s_2}u|\le \|a\|_{L^\infty(\Om)} \|u\|_{L^\infty(\Om)}^{\de-1}+\|b\|_{L^\infty(\Om)} \|u\|_{L^\infty(\Om)}^{r-1}:=K>0.\]
  Thus, using Theorem \ref{oscl} and standard covering arguments, we can show that there exists $\al\in (0, 1)$ such that $u\in C^{0,\al}(\oline{\Om^\prime})$ for all $\Omega^\prime \Subset \Om.$ \QED
   
\section{Fibering map analysis and Nehari manifold}
 \setcounter{equation}{0}
 \noindent In this section we study the fibering maps and the Nehari manifold associated to the problem $(\mc P_\la)$.  The Nehari Manifold $\mc N_\la$ defined as
 \begin{equation*}
 \mc {N_\la}=\{{u\in X_{p,s_1}\setminus\{0\}}: \langle \mc {J_\la}^\prime(u), u\rangle =0\},
 \end{equation*}
 where $\ds\langle ._,. \rangle$ is the duality between $X_{p,s_1}$ and its dual space. Clearly, $\mc  N_{\la}$ contains every  solution of the problem $(\mc P_\la)$. From the definition of $\mc N_\la$ it is clear that $u\in \mc  N_{\la}$ if and only if
   \begin{equation*}
     \int_{Q}  \frac{|u(x)-u(y)|^{p}}{|x-y|^{n+ps_1}} dxdy +\ba \int_{Q}  \frac{|u(x)-u(y)|^{q}}{|x-y|^{n+qs_2}} dxdy - \la \int_{\Om}a(x) |u|^{\de}dx- \int_{\Om} b(x)|u|^{r}dx =0.
   \end{equation*}
 Consider the fibering map for the functional $\mc {J}_\la$ which were introduced by Drabek and Pohozaev in \cite{DP}. For $u \in X_{p,s_1}$ we define $\psi_u: \mb R^+\ra \mb R$
 as $\psi_{u}(t)=\mc {J}_\la(tu)$ that is,
   \begin{align}
    \psi_{u}(t) &= \frac{t^{p}}{p} \|u\|_{X_{p,s_1}}^{p} +\ba \frac{t^{q}}{q} \|u\|_{X_{q,s_2}}^{q}- \frac{\la
     t^{\de}}{\de}\int_{\Om} a(x)|u|^{\de} dx - \frac{t^{r}}{r}\int_{\Om} b(x) |u|^{r} dx, \nonumber \\
     \label{eqb13}
    \psi_{u}^{\prime}(t) &=t^{p-1}\|u\|_{X_{p,s_1}}^{p} +\ba t^{q-1} \|u\|_{X_{q,s_2}}^{q}- {\la
    t^{\de-1}}\int_{\Om} a(x) |u|^{\de} dx  - t^{r-1}\int_{\Om} b(x) |u|^{r} dx,\\\label{eqb14}
    \psi_{u}^{\prime\prime}(t) &= (p-1)t^{p-2}\|u\|_{X_{p,s_1}}^{p}+ \ba (q-1)t^{q-2}\|u\|_{X_{q,s_2}}^{q} - (\de-1) \la
    t^{\de-2} \int_{\Om} a(x) |u|^{\de} dx \\
    &\qquad - (r-1) t^{r-2}\int_{\Om} b(x) |u|^r dx \nonumber.
   \end{align}
  From the above equations we observe that $tu\in \mc  N_{\la}$ if and only if
  $\psi_{u}^{\prime}(t)=0$ and in particular, $u\in \mc N_{\la}$ if
  and only if $\psi_{u}^{\prime}(1)=0$. Thus it is natural to split $\mc  N_{\la}$ into three parts corresponding to local minima, local maxima and points of inflection, namely
  \begin{align*}
   \mc N_{\la}^{+}= \left\{u\in \mc N_{\la}: \psi_{u}^{\prime\prime}(1) >0\right\},\;
   \mc N_{\la}^{-}= \left\{u\in \mc N_{\la}: \psi_{u}^{\prime\prime}(1) <0\right\} \text{and}~
   \mc N_\la^0= \left\{u\in \mc N_{\la}: \psi_{u}^{\prime\prime}(1) = 0\right\}.\end{align*}
  \begin{Lemma}
  If $u$ is a minimizer of $\mc{J}_{\la}$ on $\mc{N}_{\la}$ and $u \notin \mc{N}_{\la}^{0}.$ Then $u$ is a critical point of $\mc J_\la.$
 \end{Lemma}
 \begin{proof}
 The details of the proof can be found in \cite{DP}.\QED
 \end{proof}

 \begin{Lemma}\label{le44}
  $\mc{J}_{\la}$ is coercive and bounded below on  $\mc {N}_{\la}$.
 \end{Lemma}
 \begin{proof}
  In view of H\"older's inequality and the fact that $u \in \mc {N}_{\la},$ we have
    \begin{align*}
      \mc {J}_{\la}(u) &= \left(\frac{1}{p}-\frac{1}{r}\right)\|u\|_{X_{p,s_1}}^{p} +\ba \left(\frac{1}{q}-\frac{1}{r}\right)\|u\|_{X_{q,s_2}}^{q}-\la\left(\frac{1}{\de}-\frac{1}{r}\right)\int_\Omega a(x)|u|^{\de}dx,\\
      &\geq \left(\frac{1}{p}-\frac{1}{r}\right)\|u\|_{X_{p,s_1}}^{p}-\la \left(\frac{1}{\de}-\frac{1}{r}\right)\|a\|_{L^\frac{r}{r-\de}(\Om)}S_{r}^{\frac{-\de}{p}} \|u\|_{X_{p,s_1}}^{\de}.
     \end{align*}
 Thus, $\mc {J}_\la$ is coercive and bounded below in $\mc {N}_\la$.
  \end{proof}\QED
 \noindent Define \begin{align*}
 	      \theta_{\la} := \inf\{\mc {J}_{\la}(u) \ | \ u \in \mc {N}_{\la}\} \quad \mbox{and}\ \ \theta_{\la}^\pm := \inf\{\mc {J}_{\la}(u) \ | \ u \in \mc {N}_{\la}^\pm\}.
     \end{align*}
 \begin{Lemma}\label{b5}
 There exists $\la_{0} > 0$ such that $\mc {N}_{\la}^{0} = \emptyset,$ for all $\la \in (0, \la_{0})$.
 \end{Lemma}
 \begin{proof}
 We divide the proof into two cases.\\
 \textbf{Case 1:}  $u \in \mc {N}_{\la}$ and $\ds\int_{\Omega}a(x) |u|^{\de}dx = 0.$\\
  Since $u\in\mc N_\la$, we have\;$\|u\|_{X_{p,s_1}}^{p} +\ba \|u\|_{X_{q,s_2}}^{q}- \ds\int_{\Om} b(x)|u|^{r}dx =0$. Hence,
 \begin{align*}
   (p-1)\|u\|_{X_{p,s_1}}^{p}+ \ba (q-1)\|u\|_{X_{q,s_2}}^{q}-&(r-1) \ds\int_{\Om} b(x) |u|^{r} dx \\
   &=(p-r)\|u\|_{X_{p,s_1}}^{p}+ \ba (q-r)\|u\|_{X_{q,s_2}}^{q}<0, 
  \end{align*}
 \noi that is, $u \notin \mc {N}_{\la}^{0}$.\\
   \textbf{Case 2:} $u \in \mc {N}_{\la}$ and $\ds\int_{\Omega}a(x) |u|^{\de}dx \neq 0.$
  \vspace*{.1 cm}\\
  Suppose $u \in \mc {N}_{\la}^{0}$, then \eqref{eqb13} and \eqref{eqb14}, implies
    \begin{eqnarray}\label{eqb15}
    (p-\de)\|u\|_{X_{p,s_1}}^{p} + \ba (q-\de)\|u\|_{X_{q,s_2}}^{q} &=& (r-\de)\int_\Omega b(x)|u|^{r} dx,\\\label{eqb16}
   (r-p)\|u\|_{X_{p,s_1}}^{p} +\ba (r-q)\|u\|_{X_{q,s_2}}^{q} &=& (r-\de) \la \int_\Omega a(x)|u|^{\de} dx.
  \end{eqnarray}
  Define $\mc E_{\la}: \mc N_{\la} \rightarrow \mb R$ as
  \begin{equation*}
    \mc E_{\la}(u) = \frac{(r-p)\|u\|_{X_{p,s_1}}^{p} +\ba (r-q)\|u\|_{X_{q,s_2}}^{q}}{(r-\de)} - \la\int_\Omega a(x)|u|^{\de} dx.
  \end{equation*}
 { Then with the help of \eqref{eqb16} we infer $\mc E_{\la}(u) = 0$ for all $u\; \in \mc {N}_{\la}^{0}$.  Moreover,}
  \begin{align*}
 \mc E_{\la}(u) & \geq \left(\frac{r-p}{r-\de}\right)\|u\|_{X_{p,s_1}}^{p} - \la \int_\Omega a(x)|u|^{\de} dx\\
  & \geq \left(\frac{r-p}{r-\de}\right)\|u\|_{X_{p,s_1}}^{p} - \la \|a\|_{L^\frac{r}{r-\de}(\Om)}S_{r}^{\frac{-\de}{p}}\|u\|_{X_{p,s_1}}^{\de},\\
  & \geq \|u\|_{X_{p,s_1}}^{\de}\left[\left(\frac{r-p}{r-\de}\right)\|u\|_{X_{p,s_1}}^{p-\de} - \la \|a\|_{L^\frac{r}{r-\de}(\Om)}S_{r}^{\frac{-\de}{p}}\right].
  \end{align*}
  With the help of \eqref{eqb15} and H\"older inequality, we have
  \begin{equation*}
  \|u\|_{X_{p,s_1}} \geq \left(\frac{(p-\de)S_{r}^{\frac{r}{p}}}{(r-\de)\|b\|_{L^\infty(\Om)}}\right)^{\frac{1}{r-p}},
  \end{equation*}
  as a result
 \begin{equation*}
    \mc E_{\la}(u) \geq \|u\|_{X_{p,s_1}}^{\de}\left(\left(\frac{r-p}{r-\de}\right)\left(\frac{(p-\de)S_{r}^{\frac{r}{p}}}{(r-\de)\|b\|_{L^\infty(\Om)}}\right)^{\frac{p-\de}{r-p}} - \la \|a\|_{L^\frac{r}{r-\de}(\Om)}S_{r}^{\frac{-\de}{p}}\right).  \end{equation*}
 It implies that there exists 
   \begin{equation*}
     \la_0:=\;\left(\frac{(r-p)S_{r}^{\frac{\de}{p}}}{(r-\de)\|a\|_{L^\frac{r}{r-\de}(\Om)}}\right)\left(\frac{(p-\de)S_{r}^{\frac{r}{p}}}{(r-\de)\|b\|_{L^\infty(\Om)}}\right)^{\frac{p-\de}{r-p}} >0
   \end{equation*}  
 \noi such that $\mc E_{\la}(u)>0$ for all $\la\in (0, \la_0)$ and $ u \in \mc {N}_{\la}^{0}$, which contradicts the fact $\mc E_\la(u)=0$ for all  $u\in\mc N_\la^0$. Therefore, $\mc {N}_{\la}^{0} = \emptyset.$
\QED \end{proof}

  For fixed $u\in X_{p,s_1}\setminus\{0\}$, define $ M_{u}: \mb R^{+} \lra \mb R$ as
   \begin{equation*}
       M_{u}(t)= t^{p-\de}\|u\|_{X_{p,s_1}}^{p} +\ba \; t^{q-\de}\|u\|_{X_{q,s_2}}^{q}-  t^{r-\de}\int_{\Om} b(x) |u|^{r}dx.
   \end{equation*}
  Then,
     \begin{align*}
       M_{u}^{\prime}(t)&= (p-\de)t^{p-\de-1}\|u\|_{X_{p,s_1}}^{p} +\ba \; (q-\de)t^{q-\de-1}\|u\|_{X_{q,s_2}}^{q} - (r-\de)t^{r-\de-1}\int_{\Om} b(x) |u|^{r} dx
     \end{align*}
 and for $t>0$, $tu\in \mc N_{\la}$ if and only if $t$ is a solution of
 $M_{u}(t)={\la} \ds\int_{\Om} a(x) |u|^{\de} dx$. Moreover, if $tu\in\mc N_\la$, then $\psi_{tu}^{\prime \prime}(1)=t^{\de+1}M_{u}^{\prime}(t)$.
  Now we study the fibering map $\psi_{u}$ according to the sign of
 $\ds\int_{\Om} a(x)|u|^{\de} dx$ and $\ds\int_{\Om} b(x)|u|^{r} dx$.

 \noi \textbf{Case 1}: If $\ds\int_{\Om} a(x)|u|^{\de} dx>0$ and $\ds\int_{\Om} b(x)|u|^{r} dx>0$.\\
 We see  $M_u(t)\ra -\infty$ as $t \ra \infty$, $M_u(t)>0$ for $t$ small enough and $M^{\prime}_u(t)<0$ for $t$ large enough.  We claim that there exists unique $t_{max}>0$ such that $M_u^\prime(t_{max})=0$. 
	\begin{align*}
		M_u^\prime(t)= &(p-\de)t^{p-\de-1}\|u\|_{X_{p,s_1}}^{p} +\ba \; (q-\de)t^{q-\de-1}\|u\|_{X_{q,s_2}}^{q} - (r-\de)t^{r-\de-1}\int_{\Om} b(x) |u|^{r} dx\\
		  =&t^{q-\de-1}\left((p-\de)t^{p-q}\|u\|_{X_{p,s_1}}^{p} +\ba \; (q-\de)\|u\|_{X_{q,s_2}}^{q} - (r-\de)t^{r-q}\int_{\Om} b(x) |u|^{r} dx\right).
	\end{align*}
 Let $G_u(t)=(p-\de)t^{p-q}\|u\|_{X_{p,s_1}}^{p} +\ba \; (q-\de)\|u\|_{X_{q,s_2}}^{q} - (r-\de)t^{r-q}\ds\int_{\Om} b(x) |u|^{r} dx$, then to prove the claim it is enough to show existence of unique $t_{max}>0$ satisfying $G_u(t_{max})=0$. Define  $H_u(t)= (r-\de)\; t^{r-q}\ds\int_{\Om} b(x) |u|^{r} dx -(p-\de)t^{p-q}\|u\|_{X_{p,s_1}}^{p}$, then $H_u(t_{max})- \ba\;(q-\de)\|u\|_{X_{q,s_2}}^q= -G_u(t)$. It is easy to see  $H_u(t)<0$ for $t$ small enough, $H_u(t)\ra\infty$ as $t\ra\infty$. Hence, there exist unique $t_*>0$ such that $H_u(t_*)=0$. Indeed $t_*=\left(\frac{(p-\de)\|u\|_{X_{p,s_1}}^p}{(r-\de)\ds\int_\Om b(x)|u|^r}\right)^\frac{1}{r-p}>0$.  Therefore, there exists unique $t_{max}>t_*>0$ such that $H_u(t_{max})= \ba\;(q-\de)\|u\|_{X_{q,s_2}}^q$.
 Moreover, $M_u$ is increasing in $(0,t_{max})$, decreasing in $(t_{max},\infty)$. As a consequence
   \begin{align*}
     (p-\de)t^{p}_{max}\|u\|_{X_{p,s_1}}^{p}&\leq (p-\de)t^{p}_{max}\|u\|_{X_{p,s_1}}^{p}+ \ba (q-\de)t^{p_2}_{max}\|u\|_{X_{q,s_2}}^{p_2} \\&=(r-\de) t^{r}_{max}\int_{\Om} b(x) |u|^{r} dx\leq (r-\de) t^{r}_{max}\|b\|_{L^\infty(\Om)}S_{r}^{\frac{-r}{p}}\|u\|_{X_{p,s_1}}^{r},
   \end{align*}
 define
     \begin{equation*}
       T_0 :=  \frac{1}{\|u\|_{X_{p,s_1}}}\left(\frac{(p-\de)S_{r}^{\frac{r}{p}}}{(r-\de)\|b\|_{L^\infty(\Om)}}\right)^{\frac{1}{r-p}}\leq t_{max},
    \end{equation*}
   then,
     \begin{equation*}
      \begin{aligned}
         M_u(t_{max})\geq  M_u(T_0)&\geq T_0^{p-\de} \|u\|^{p}_{X_{p,s_1}} -T_0^{r-\de} \|b\|_{L^\infty(\Om)} S^{\frac{-r}{p}}\|u\|^{r}_{X_{p,s_1}}\\&=\|u\|^{\de}_{X_{p,s_1}}\left(\frac{r-p}{r-\de}\right)\left(\frac{(p-\de)S_{r}^{\frac{r}{p}}}{(r-\de)\|b\|_{L^\infty(\Om)}}\right)^{\frac{p-\de}{r-p}}\geq 0.
      \end{aligned}
    \end{equation*}
  Therefore, if $\la < \la_0$, we have $M_u(t_{max})>\la\ds\int_\Om a(x)|u|^\de dx$, which ensures the existence of $t_1<t_{max}$ and $t_2>t_{max}$ such that $M_u(t_1)=M_u(t_2)=\la \ds\int_{\Om} a(x) |u|^{\de} dx $. That is,  $t_1u$ and $t_2u \in \mc {N}_{\la}.$ Also, $M_u^{\prime}(t_1)>0$ and $M_u^{\prime}(t_2)<0$ implies $t_1u \in \mc{N}^{+}_{\la}$ and $t_2u \in \mc {N}^{-}_{\la}.$  By the fact $\psi^{\prime}_{u}(t) = t^{\de}\left(M_{u}(t)- \la\ds\int_{\Om} a(x)|u|^{\de}\right)$, we deduce that $\psi^{\prime}_{u}(t)<0$ for all $t \in [0, t_1)$ and $\psi^{\prime}_{u}(t)>0$ for all $t \in (t_1, t_2)$. So, $\mc {J}_\la(t_1u) = \displaystyle\min_{0 \leq t \leq t_2}\mc {J}_\la (tu).$ Moreover, $\psi^{\prime}_u(t) > 0$ for all $t \in [t_1, t_2),\;
  \psi^{\prime}_u(t_2) = 0$ and $\psi^{\prime}_u(t) < 0$ for all $t \in (t_2, \infty)$ implies  $\mc {J}_\la(t_2u)
  = \displaystyle\max_{t \geq t_{\max}} \mc {J}_\la(tu).$\\
  \noi \textbf{ Case 2}: If $\ds\int_{\Om} a(x)|u|^{\de} dx<0$ and $\ds\int_{\Om} b(x)|u|^{r} dx>0$.\\
 Since $M_{u}(t)\ra -\infty$ as $t \ra \infty$, $M_u(t)>0$ for $t$ small enough and  $M^{\prime}_u(t)<0$ for $t$ large enough, by the same assertions as in case (1) there exists unique $t_0>0$ such that $M_{u}$ is increasing in $(0,t_0)$, decreasing in $(t_0,\infty)$ and $M_{u}^{\prime}(t_0)=0$. Taking into account the fact $M_{u}(t_0)>0$ and $\la\ds\int_{\Om} a(x)|u|^{\de} dx<0$, we get unique $t_1$ such that $M_u(t_1)=\la \ds\int_{\Om} a(x) |u|^{\de} dx $ and $M_{u}^{\prime}(t_1)<0$ which implies $t_1u\in \mc {N}_{\la}^{-}$ that is, $t_1u$ is a local maximum.\\
 \noi \textbf{Case 3}: If $\ds\int_{\Om} a(x)|u|^{\de} dx>0$ and $\ds\int_{\Om} b(x)|u|^{r} dx<0$.\\
 In this case $M_{u}^{\prime}(t)>0$ for all $t>0$, this implies $M_{u}$ is an increasing function. Therefore, there exist unique $t_1>0$ such that $M_{u}(t_1)={\la}\ds \int_{\Om} a(x) |u|^{\de} dx$ with $\psi^{\prime \prime}_{t_1u}(1)>0$. So, $t_1u \in \mc {N_{\la}^+}$ that is, $t_1u $ is a local minimum.\\
 \noi \textbf{Case 4}: If $\ds\int_{\Om} a(x)|u|^{\de} dx<0$ and $\ds\int_{\Om} b(x)|u|^{r} dx<0$.\\
 In this case $\psi_{u}(0)=0$ and $\psi_{u}^{\prime}(t)>0$ for all $t>0$, which implies that $\psi_{u}$ is strictly increasing and hence has no critical point.
 \begin{Lemma}\label{L35}
  There exists a constant $C_{2}>0$ such that
       \begin{equation*}
          \theta_{\la}^+ \leq - \frac{(p-\de)(r-p)}{p\;\de\;r}\;C_{2}<0.
       \end{equation*}
 \end{Lemma}
 \begin{proof}
  Let $u_0 \in \mathrm{X}_{p}$ be such that $\ds\int_\Om{a(x)|u_0|^{\de}dx}>0$. Then there exists $ t_0 > 0$ such that  $t_0u_0 \in \mc{N}^{+}_{\la}$ that is, $\psi_{t_{0}u_{0}}^{\prime\prime}(1)>0.$ Hence using \eqref{eqb13} and \eqref{eqb14}, we have
    \begin{align*}
       \mc {J}_{\la}(t_{0}u_{0}) &\leq \left(\frac{p-\de}{\de}\right) \left(\frac{1}{r}-\frac{1}{p}\right) \|t_{0} u_{0}\|_{X_{p,s_1}}^{p}+ \ba \left(\frac{q-\de}{\de}\right) \left(\frac{1}{r}-\frac{1}{q}\right) \|t_{0} u_{0}\|_{X_{q,s_2}}^{q}\\
      &\leq -\frac{(p-\de)(r-p)}{p\;\de\;r} \; C_2 
    \end{align*}
   where $C_2=\|t_0 u_0 \|_{X_{p,s_1}}^{p}.$ This implies $ \theta_{\la}^+ \leq -\frac{(p-\de)(r-p)}{p\;\de\;r}\;C_{2}<0$.\QED
  \end{proof}
 \begin{Lemma}\label{tt}
  Let $\la \in (0, \la_{0})$ and ${z \in \mc {N}_{\la}}$. Then there exist $\epsilon > 0$ and a differentiable function
  $\xi : \mc {B}(0,\epsilon) \subseteq X_{p} \rightarrow \mathbb{R}^{+}$ such that $\xi(0)=1,$  $\xi(w)(z-w)\in \mc {N}_{\la}$
 and
  \begin{equation}\label{eqb18}
      \langle\xi^{\prime}(0), w\rangle = \frac{p A_p(z, w) +\ba q A_q(z, w)- \de\la\ds \int_\Omega a(x)|z|^{\de-2}z\;w\;dx- r\ds\int_\Om b(x)|z|^{r-2}z\;w dx}{(p-\de)\|z\|_{X_{p,s_1}}^{p}+ \ba (q-\de)\|z\|_{X_{q,s_2}}^{q} -(r-\de)\ds\int_\Omega b(x)|z|^{r}dx}
  \end{equation}
  for all $w\in X_{p,s_1}$.
 \end{Lemma}
 \begin{proof}
  For  ${z\in \mc {N}_\la}$ define  a function $\mc {H}_z:\mathbb{R}\times X_{p} \rightarrow \mathbb{R}$ given by
  \begin{equation*}
  \begin{aligned}
    \mc {H}_z(t,w) &:= \langle J^{\prime}_{\la}(t(z-w)),t(z-w)\rangle\\& \;=
      t^{p}\|z-w\|_{X_{p,s_1}}^{p}+ \ba t^{q}\|z-w\|_{X_{q,s_2}}^{q}-t^{\de}\la \int_\Omega{a(x)|z-w|^{\de}dx}-t^{r}\int_\Omega b(x){|z-w|^{r}dx}.
   \end{aligned}
   \end{equation*}
\noi Then, $\mc {H}_z(1,0) = \langle J^{\prime}_{\la} (z),z\rangle = 0$ and by Lemma \ref{b5}, we deduce that 
\begin{equation*}
\frac{\partial}{ \partial t}\mc {H}_z(1,0)= (p-\de)\|z\|^{p}_{X_{p,s_1}}+ \ba (q-\de)\|z\|^{q}_{X_{q,s_2}} -(r-\de)\int_{\Om}b(x)|z|^{r}dx \neq 0.
\end{equation*}
 \noi Now, by implicit function theorem there exist $\epsilon >0$ and a  differentiable function $\xi : \mc {B}(0, \epsilon) \subseteq \mathrm{X}_{p} \rightarrow \mathbb{R}$ such that $\xi(0) = 1$, \eqref{eqb18} holds and $\mc {H}_u(\xi(w),w) = 0$ $\textrm{for all}\; w \in \mc {B}(0, \epsilon)$,
which is equivalent to 

\begin{equation*}
\begin{aligned}
 0&= \|\xi(w)(z-w)\|_{X_{p,s_1}}^{p} +\ba \|\xi(w)(z-w)\|_{X_{q,s_2}}^{q}- \la \int_\Omega{a(x)|\xi(w)(z-w)|^{\de}dx} \\&\quad- \int_\Omega b(x){|\xi(w)(z-w)|^{r}dx} 
 \end{aligned}
\end{equation*}
 for all $w \in \mc {B}(0, \e)$. Hence $ \xi(w)(z-w) \in \mc {N}_{\la}$.
 \QED
 \end{proof}

\section{Multiplicity results}
In this section we prove existence and multiplicity of non-trivial solutions of problem $(\mc P_\la)$ for the case $\de<q\leq p<r\leq p^*_{s_1}$.  
 \begin{Proposition}\label{propb2}
	Let $\la \in (0,\la_{0})$, then there exists a sequence $\{u_k\} \subset \mc {N}_{\la}$ such that
	\begin{center}
		$\mc {J}_{\la}(u_{k}) = \theta_{\la}+o_k(1),$ and $\mc {J}_{\la}^{\prime}(u_{k}) = o_k(1).$
	\end{center}
\end{Proposition}
\begin{proof}
	Using Lemma \ref{le44} and  Ekeland variational principle \cite{eke1974}, there exists a minimizing sequence $\{u_k\}\subset\mc {N}_\la $ such that
	\begin{equation}\label{evp1}
	\mc {J}_\la(u_k)< \theta_\la+\frac{1}{k}, \mbox{ and}
	\end{equation}
	\begin{equation}\label{eqb20}
	\mc {J}_\la(u_k)< \mc {J}_\la(v)+\frac{1}{k}\|v-u_k\|_{X_{p,s_1}}\; \textrm{for each}\; v \in \mc {N}_\la.
	\end{equation}
	For large $k$, using \eqref{evp1}, we have
	\begin{align*}
	\mc J_{\la}(u_k)&= \left(\frac{1}{p}-\frac{1}{r}\right)\|u\|_{X_{p,s_1}}^{p} +\ba \left(\frac{1}{q}-\frac{1}{r}\right)\|u\|_{X_{q,s_2}}^{q}-\la\left(\frac{1}{\de}-\frac{1}{r}\right)\int_\Omega a(x)|u|^{\de}dx\\
	&< \theta_\la+\frac{1}{k}<\theta_\la^{+}<0.
	\end{align*}
	Therefore, 
	$$\left(\frac{1}{p}-\frac{1}{r}\right)\|u\|_{X_{p,s_1}}^{p} -\la\left(\frac{1}{\de}-\frac{1}{r}\right)\int_\Omega a(x)|u|^{\de}dx <0,$$
	which implies $u_k\not\equiv 0$ for large $k$.
	Then, using H\"older's inequality, we obtain 
	\begin{equation}\label{eqb21}
	\left(\frac{\;\la(-\theta_{\la}^{+})\;\de\;rS_{r}^{\frac{\de}{p}}}{(r-\de)\|a\|_{L^\frac{r}{r-\de}(\Om)}}\right)^{\frac{1}{\de}}\leq \|u\|_{X_{p,s_1}}\leq \left(\frac{\;\la p(r-\de)\|a\|_{L^\frac{r}{r-\de}(\Om)}}{\de(r-p)S_{r}^{\frac{\de}{p}}}\right)^{\frac{1}{p-\de}}.
	\end{equation}
	\noi Now, we prove $\|\mc {J}^{\prime}_\la(u_k)\|\rightarrow 0$ as $k\rightarrow \infty$. Employing Lemma \ref{tt} for each $u_k$ we obtain differentiable functions $\xi_k:\mc {B}(0, \epsilon_k)\rightarrow \mathbb{R}$ for some $\epsilon_k>0$ such that $\xi_k(v)(u_k-v)\in \mc {N}_\la$,\; $\textrm{for all}\; v\in \mc {B}(0, \epsilon_k).$
	\noindent Fix $k\in\mb N$, $u(\not\equiv 0) \in X_{p,s_1}$ and $0<\rho<\epsilon_k$. Set $v_\rho=\frac{\rho u}{\|u\|_{X_{p,s_1}}}$ and $h_\rho=\xi_k(v_\rho)(u_k-v_\rho)$, then $h_\rho \in \mc {N}_\la$. Therefore, from \eqref{eqb20}, we have
	\begin{align*}
	\mc {J}_\la(h_\rho)-\mc {J}_\la(u_k)\geq-\frac{1}{k}\|h_\rho-u_k\|_{X_{p,s_1}}
	\end{align*}
	which on using mean value theorem gives us 
	\begin{align*}
	\langle J^{\prime}_{\la}(u_k),h_\rho-u_k \rangle+ o\big(\|h_\rho-u_k\|_{X_{p,s_1}}\big)\geq -\frac{1}{k}\|h_\rho-u_k\|_{X_{p,s_1}}.
	\end{align*}
	It implies that
	{\small\begin{align*}
		-\rho\left\langle J^{\prime}_{\la}(u_k),\frac{u}{\|u\|_{X_{p,s_1}}} \right\rangle +(\xi_k(v\rho)-1)\langle J^{\prime}_{\la}(u_k),u_k-v\rho \rangle\geq -\frac{\|h_\rho-u_k\|_{X_{s_1}}}{k}+o\big(\|h_\rho-u_k\|_{X_{p,s_1}}\big).
		\end{align*}\small}
	Thus, using $h_\rho\in\mc N_\la$, we deduce that
	 \begin{align*}
		\left\langle J^{\prime}_{\la}(u_k),\frac{u}{\|u\|_{X_{p,s_1}}} \right\rangle &\leq  \frac{\|h_\rho-u_k\|_{X_{p,s_1}}}{k \rho}+\frac{o(\|h_\rho-u_k\|_{X_{p,s_1}})}{\rho} \\ & \qquad+\frac{(\xi_k(v\rho)-1)}{\rho}\left\langle J^{\prime}_{\la}(u_k)-J^{\prime}_{\la}(h\rho),u_k-v\rho \right\rangle.
		\end{align*}
	\noi Since $\ds\lim_{k\rightarrow \infty}\frac{|\xi_k(v_\rho)-1|}{\rho}\leq \|\xi_k(0)\|_{X_{p,s_1}}$ and $\|h_\rho-u_k\|_{X_{p,s_1}}\leq \rho|\xi_k(v_\rho)|+|\xi_k(v_\rho)-1|\|u_k\|_{X_{p,s_1}}$, for  fixed $k$ if $\rho \ra 0$, by \eqref{eqb21}, there exists a constant $C>0$ independent of $\rho$ such that 
	\begin{equation*}
	\left\langle \mc {J}^{\prime}_\la(u_k),\frac{u}{\|u\|_{X_{p,s_1}}}\right\rangle\leq\frac{C}{k}(1+\|\xi_k^{\prime}(0)\|_{X_{p,s_1}}).
	\end{equation*}
	Thus, in order to complete the proof it is sufficient to prove $\|\xi_k^{\prime}(0)\|_{X_{p,s_1}}$ is bounded. Using \eqref{eqb18} and \eqref{eqb21}, we infer that
	\begin{equation*}
	|\langle \xi_k^{\prime}(0), v\rangle|\leq \frac{K\|v\|_{X_{p,s_1}}}{|(p-\de)\|u_k\|_{X_{p,s_1}}^{p} +\ba (q-\de)\|u_k\|_{X_{q,s_2}}^{q}-(r-\de)\ds\int_\Omega b(x)|u_k|^{r}dx|}
	\end{equation*}
	 for some $K>0$. We claim that
	\; $(p-\de)\|u_k\|_{X_{p,s_1}}^{p} +\ba (q-\de)\|u_k\|_{X_{q,s_2}}^{q}-(r-\de)\ds\int_\Omega b(x)|u_k|^{r}dx$ is bounded away from zero. On the contrary suppose there exists a subsequence of $\{u_k\}$ (still denoting by $\{u_k\}$) such that
	\begin{equation} \label{baw}
	(p-\de)\|u_k\|_{X_{p,s_1}}^{p} +\ba (q-\de)\|u_k\|_{X_{q,s_2}}^{q}-(r-\de)\int_\Omega b(x)|u_k|^{r}dx=o_k(1).
	\end{equation}
	Then, from \eqref{baw} and the fact that $u_k\in\mc N_\la$, we get $\mc E_\la(u_k) = o_k(1)$. Moreover, 
	\begin{equation*}
	\|u_k\|_{X_{p,s_1}} \geq  \left(\frac{(p-\de)S_{r}^{\frac{r}{p}}}{(r-\de)\|b\|_{L^\infty(\Om)}}\right)^{\frac{1}{r-p}}+o_k(1).
	\end{equation*}
	This gives that there exists a positive constant $d$ such that $0< d\leq \|u_k\|_{X_{p,s_1}}$ for $k$ large.
	Now following the proof of Lemma \ref{b5}, we get $\mc E_\la (u_k)>0$ for large $k$, which is not possible. 
	Therefore claim holds, which shows that $\mc{J}^\prime_\la(u_k) \ra 0$ as $k\ra\infty$.\QED
\end{proof}
\begin{Lemma}\label{lemmab3}
If $r<p^*_{s_1}$  then every Palais-Smale sequence of $\mc {J}_\la$ has a convergent subsequence. That is,
if $\{u_k\}\subset X_{p,s_1}$ satisfies
\begin{equation}\label{eqb17}
\mc {J}_\la(u_k)=c+o_k(1) \;\text{and}\; \mc {J}'_\la(u_k)=o_k(1)\; \textrm{in}\; X_{p,s_1}^{\prime},
\end{equation}
 then $\{u_k\}$ has a convergent subsequence in $X_{p,s_1}$.
\end{Lemma}
\begin{proof}
Let $\{u_{k}\}\subset X_{p,s_1}$ be a sequence satisfying \eqref{eqb17}. By standard arguments we can show that $\{u_{k}\}$ is bounded in $X_{p}$. So, we can assume there exists $u_\la\in X_{p,s_1}$ such that upto subsequence $u_{k} \rightharpoonup u_{\la}$ weakly in $\mathrm{X}_{p}$, $u_{k} \rightarrow u_\la$ strongly in $\mathrm{L}^{\ga}(\Omega)$, for $1 \le \ga < p^*_{s_1}$ and $u_{k}(x) \rightarrow u_\la(x)$ \text{a.e.} in $\Omega$. Since $\langle\mc {J}_{\la}^{\prime}(u_k)-\mc {J}_{\la}^{\prime}(u_\la), (u_{k}-u_\la)\rangle \ra 0$, as $k\rightarrow \infty$, we deduce that
\begin{align*}
    o_k(1)= &\langle\mc {J}_\la^\prime (u_k)-\mc {J}_\la^\prime(u_\la), u_k-u_\la\rangle \\
     = &A_p(u_k, u_k-u_\la)- A_p(u_\la, u_k-u_\la)+\ba (A_q(u_k, u_k-u_\la)- A_q(u_\la, u_k-u_\la) )\\
    & -\la
   \int_{\Om} a(x)\big(|u_k(x)|^{\de-2}u_k(x)-|u_\la(x)|^{\de-2}u_\la(x)\big)(u_k(x)-u_\la(x))dx \\ &-\int_{\Om} b(x)\big(|u_k(x)|^{r-2}u_k(x)-|u_\la(x)|^{r-2}u_\la(x)\big)(u_k(x)-u_\la(x))dx.
\end{align*}
By H\"older's inequality, it follows that
\begin{equation*}
\int_\Omega{a(x)|u_{k}|^{\de-2}u_k|u_k-u_\la|dx} \leq \|a\|_{L^\frac{r}{r-\de}(\Om)}\|u_k\|_{L^r(\Om)}^{\de-1}\|u_k-u_\la\|_{L^r(\Om)}\rightarrow 0\;\;\mathrm{as}\;\; k \ra \infty \mbox{ and}
\end{equation*}
\noi similarly, $\ds\int_\Omega{b(x)|u_{k}|^{r-2}u_k|u_k-u_\la|dx} \ra 0$ as $k\ra\infty$.
Now we divide the proof into three cases \\
\noi \textbf{Case 1:} If $p, q\ge 2$.\\ 
Using the inequality $|a-b|^{l} \leq 2^{l-2}(|a|^{l-2}a-|b|^{l-2}b)(a-b) \;\text{for}\; a, b \in \mathbb{R}^{n}\text{ and } l \geq 2,$ we obtain
\begin{equation*}
	o_k(1)=\langle\mc{J}_\la^\prime(u_k)-\mc{J}_\la^\prime(u_\la), u_k-u_\la\rangle  \ge
	 \|u_k-u_\la\|_{X_{p,s_1}}^{p}+ \ba \|u_k-u_\la\|_{X_{q,s_2}}^{q},
\end{equation*}
it implies that $$\|u_k-u_\la\|_{X_{p,s_1}}\ra 0 \ \text{and } \|u_k-u_\la\|_{X_{q,s_2}}\ra 0\quad \mbox{as }k\ra\infty.$$
\noi \textbf{Case 2:} If $1<q<p<2$.\\
As we know that 
for $a,\ b\in\mb{R}^n$ and $1<m<2$, there exists $C_m >0$  a constant such that
$$|a-b|^m\le C_m((|a|^{m-2}a-|b|^{m-2}b)(a-b))^{\frac{m}{2}}(|a|^m+|b|^m)^{\frac{2-m}{2}}.$$
Set $a=u_k(x)-u_k(y)$, $b=u_\la(x)-u_\la(y)$ and using H\"older inequality, we deduce that
 \begin{align*}
	\|u_k-u_\la\|_{X_{p,s_1}}^{p} 
	&\le C (A_p(u_k, u_k-u_\la) - A_p(u_\la, u_k-u_\la) )^\frac{p}{2} &\\
	&\qquad\quad \left(\ds\int_Q\frac{|u_k(x)-u_k(y)|^{p}+|u_\la(x)-u_\la(y)|^{p}}{|x-y|^{n+ps_1}} \right)^\frac{2-p}{2} 
 \end{align*}
and boundedness of $\{u_k\}$ in $X_{p,s_1}$, implies 
$$\|u_k-u_\la\|_{X_{p,s_1}}^{p}\leq C (A_p(u_k, u_k-u_\la) - A_p(u_\la, u_k-u_\la) )^\frac{p}{2}. $$
Thus,
\begin{align*}
	o_k(1)=\langle \mc{J}_\la^\prime(u_k)-\mc{J}_\la^\prime(u_\la), u_k-u_\la\rangle \ge \frac{1}{C}\big(\|u_k-u_\la\|_{X_{p,s_1}}^2+\ba\|u_k-u_\la\|_{X_{q,s_2}}^2\big).
\end{align*}
Hence, it concludes the proof.\\
\noi \textbf{Case 3:} If $1<q<2<p$.\\
	Coupling the arguments of case 1 and case 2 one can easily prove the convergence of the sequence. 
  \QED
\end{proof}

 \textbf{Proof of Theorem \ref{pqthm1}:} 
 Using proposition \ref{propb2}, we get minimizing sequences $\{u_k\}$ in $\mc N_\la^+$, and $\{v_k\}$ in $\mc N_\la^-$ and by lemma \ref{lemmab3}, there exist $u_\la, \ v_\la\in X_{p,s_1}$ such that $u_k\ra u_\la$ and $v_k\ra v_\la$ strongly in $X_{p,s_1}$ for all $\la\in(0,\la_0)$. Therefore, $u_\la$ and $v_\la$ are weak solutions of problem $(\mc P_\la)$. With the help of lemma \ref{L35}, we conclude $u_\la\not\equiv 0$, hence $u_\la\in\mc N_\la$. Moreover, by means of lemma \ref{b5}, $u_\la\in\mc N_\la^+$ with $\mc J_\la(u_\la)=\theta_{\la}^+$ and since $\mc N_\la^-$ is closed, $v_\la\in\mc N_\la^-$ with $\mc J_\la(v_\la)= \theta_{\la}^-$. Using the fact $\mc N_\la^+ \cap \mc N_\la^- =\emptyset$, we note that $u_\la$ and $v_\la$ are distinct.\\
 Now we prove non-negativity of $u_\la$. If $u_\la\ge 0$, then we have a non negative solution of $(\mc P_\la)$ which is also a minimizer for $\mc J_\la$ in $\mc N_\la^+$, otherwise we have $|u_\la|\not\equiv 0$, hence by fibering map analysis we get unique $t_1>0$ such that $t_1u_\la\in\mc N_\la^+$. We note that $M_{|u_\la|}(1)\leq M_{u_\la}(1)=\la\ds\int_\Om a(x)|u_\la|^\de= M_{|u_\la|}(t_1)\le M_{u_\la}(t_1)$ and $0< M^\prime_{u_\la}(1)$, because of the fact $u_\la\in\mc N_\la^+$, which implies $t_1\ge 1$. Thus,
 $$\theta_\la^+\le\psi_{|u_\la|}(t_1)\leq \psi_{|u_\la|}(1)\leq\psi_{u_\la}(1)=\theta_{\la}^+.$$
 Hence, $\mc J_\la(t_1|u_\la|)=\psi_{|u_\la|}(t_1)=\theta_\la^+$ and $t_1|u_\la|\in\mc N_\la^+$ that is, $t_1|u_\la|$ is a nonnegative solution of problem $(\mc P_\la)$ in $\mc N_\la^+$. By using similar arguments $v_\la$ is also a non-negative solution. \QED

Now we will show the existence and multiplicity of solutions of $(\mc P_\la)$ for the case $r=p^*_{s_1}$. From now onwards,  we will assume function $a(x)$ is continuous and there exists  $\kappa_1>0$ such that $m_a:=\ds\inf_ {x\in B_{\kappa_1}(0)}a(x)$ is positive, and $b(x)\equiv 1$ in $\Om$.
 \begin{Theorem}\label{thb2}
  If $r= p^*_{s_1}$ and $\{u_k\}\subset \mc{N}_{\la}$ is a $(PS)_c$ sequence for $\mc{J}_\la$ with $u_k \rightharpoonup u$ weakly in $X_{p,s_1}$, then $\mc{J}_\la^{\prime}(u)=0$ and there exists a positive constant $ C_\de$  depending on $p,s_1,n,S,|\Om|,\de$ such that $\mc J_{\la}(u)\geq - C_\de\; \la^{\frac{p}{p-\de}},$
  where
    \begin{equation}\label{eqb27}
       C_\de=\left(\frac{(p^*_{s_1}-\de)(p-\de)}{p \; \de \;p^*_{s_1}}\right) \left(\frac{p^*_{s_1}-\de}{p^*_{s_1}-p}\right)^{\frac{p\de}{p-\de}} S^{\frac{-\de}{p-\de}}\|a\|_{L^\infty(\Om)}^{\frac{p}{p-\de}} \; |\Om|^{\frac{p(p^*_{s_1}-\de)}{(p-\de)p^*_{s_1}}}.
    \end{equation} 
   \end{Theorem}
 \begin{proof}
 Since $u_k \rightharpoonup  u$ in $X_{p,s_1}$, there exists a subsequence (still denoted by $u_k$) $u_k  \ra u \text{ in } L^m(\Om),$  $1\leq m < p^*_{s_1}$ and $u_k(x) \ra u(x)$ a.e. in $\Om$.\\
 \textbf{Claim:} $\displaystyle \lim_{k \ra \infty}A_p(u_k , \phi) =A_p(u , \phi)$  and $\displaystyle \lim_{k \ra \infty}A_q(u_k , \phi) =A_q(u , \phi)$ for any $\phi \in X_{p,s_1}$.\\
 \noi The sequence $\left(\frac{|u_k(x)-u_k(y)|^{p-2}(u_k(x)-u_k(y))}{|x-y|^{\frac{n+ps_1}{p^{\prime}}}}\right)$ is bounded in  $L^{p^{\prime}}(\mathbb{R}^{2n})$, where $\frac{1}{p}+\frac{1}{p^{\prime}}=1$, thus upto subsequence, we have
\begin{align*}
   \left(\frac{|u_k(x)-u_k(y)|^{p-2}(u_k(x)-u_k(y))}{|x-y|^\frac{n+ps_1}{p^\prime}}\right){\rightharpoonup} \left(\frac{|u(x)-u(y)|^{p-2}(u(x)-u(y))}{|x-y|^\frac{n+ps_1}{p^\prime}}\right)
\end{align*}
 weakly in $ L^{p^{\prime}}(\mathbb{R}^{2n})$, which on using the fact that  $\left(\frac{\phi(x)-\phi(y)}{|x-y|^{\frac{n+ps_1}{p}}}\right)\in L^{p}(\mathbb{R}^{2n})$ gives us \\ $\ds\lim_{k \ra \infty}A_p(u_k , \phi) =A_p(u , \phi)$. Similarly, we can prove that $\ds \lim_{k \ra \infty}A_q(u_k , \phi) =A_q(u , \phi)$. This concludes the proof of the claim. By using the fact that $u_k\rightharpoonup u$ in $X_{p,s_1}$, we have
\begin{align*}
 & |u_k|^{\de-2}u_k {\rightharpoonup}|u|^{\de-2}u \quad  \text{ weakly in } L^{\de^{\prime}}(\Om) \text{ and }  \\
 & |u_k|^{p^*_{s_1}-2}u_k {\rightharpoonup}|u|^{p^*_{s_1}-2}u \text{ weakly in }	L^{(p_{s_1}^*)^\prime}(\Om).
\end{align*}
 It implies, for any $\phi \in L^\de(\Om)\;\cap\; L^{p^*_{s_1}}(\Om)$
   \begin{align*}
      &\int_{\Om}a(x)(|u_k(x)|^{\de-2}u_k(x)-|u(x)|^{\de-2}u(x))\phi(x)\;dx\ra 0,\\  &\int_{\Om}b(x)(|u_k(x)|^{p^*_{s_1}-2}u_k(x)-|u(x)|^{p^*_{s_1}-2}u(x))\phi(x)\;dx \ra 0. 
   \end{align*}

 \noi Resuming all the information collected so far, for any $\phi\in X_{p,s_1}$, we deduce that
 \begin{equation}
   \begin{aligned}
     \langle \mc {J}_\la^{\prime}(u_k)-\mc {J}_\la^{\prime}(u),\phi \rangle =& A_p(u_k,\phi)-A_p(u,\phi ) + \ba \left(A_q(u_k,\phi)-A_q( u,\phi) \right)\\ -&\int_{\Om}a(x)(|u_k(x)|^{\de-2}u_k(x)-|u(x)|^{\de-2}u(x))\phi(x)\;dx\\ -& \int_{\Om}b(x)(|u_k(x)|^{p^*_{s_1}-2}u_k(x)-|u(x)|^{p^*_{s_1}-2}u(x))\phi(x)\;dx \\ = \;&o_k(1). \nonumber
  \end{aligned}
  \end{equation}
  \noi This implies  $ \mc {J}_\la^{\prime}(u)=0$. In particular $\langle \mc {J}_\la^{\prime}(u),u\rangle=0$, which gives us
  \begin{equation}\label{eqb25}
   \begin{aligned}
     \mc {J}_{\la}(u) &= \left(\frac{1}{p}-\frac{1}{p^*_{s_1}}\right) \|u\|_{X_{p,s_1}}^{p} +\ba \left(\frac{1}{q}- \frac{1}{p^*_{s_1}}\right) \|u\|_{X_{q,s_2}}^{q}- \la\left(\frac{1}{\de}-\frac{1}{p^*_{s_1}}\right)\int_\Omega a(x)|u|^{\de}dx  \\ 
    & \geq \left(\frac{1}{p}-\frac{1}{p^*_{s_1}}\right)\|u\|_{X_{p,s_1}}^{p}-\la   \left(\frac{1}{\de}-\frac{1}{p^*_{s_1}}\right)\int_\Omega a(x)|u|^{\de}dx.
  \end{aligned}
  \end{equation}
  \noi By H\"older's inequality, Sobolev embeddings and Young inequality, we obtain
    \begin{equation}\label{eqb26}
    \begin{aligned}
     \la \int_\Omega a(x)|u|^{\de}dx &\leq \la \|a\|_{L^\infty(\Om)}S^{\frac{-\de}{p}} |\Om|^{\frac{p^*_{s_1}-\de}{p^*_{s_1}}} \|u\|_{X_{p,s_1}}^{\de}\\&= \left(\frac{p}{\de}\left(\frac{1}{p}-\frac{1}{p^*_{s_1}}\right) \left(\frac{1}{\de}-\frac{1}{p^*_{s_1}}\right)^{-1}\right)^\frac{\de}{p}\|u\|_{X_{p,s_1}}^{\de}\\&\qquad \quad \la \left(\frac{p}{\de}\left(\frac{1}{p}-\frac{1}{p^*_{s_1}}\right) \left(\frac{1}{\de}-\frac{1}{p^*_{s_1}}\right)^{-1}\right)^\frac{-\de}{p}\|a\|_{L^\infty(\Om)} \; |\Om|^{\frac{p^*_{s_1}-\de}{p^*_{s_1}}} S^{\frac{-\de}{p}} \\& \leq \left(\frac{1}{p}-\frac{1}{p^*_{s_1}}\right) \left(\frac{1}{\de}- \frac{1}{p^*_{s_1}}\right)^{-1}\|u\|_{X_{p,s_1}}^{p}+ A \la ^{\frac{p}{p-\de}},
  \end{aligned}
  \end{equation}
 \noi where $A=  \left(\frac{p-\de}{p}\right)\left(\frac{p^*_{s_1}-\de}{p^*_{s_1}-p}\right)^{\frac{p\de}{p-\de}}  S^{\frac{-\de}{p-\de}}\|a\|_{L^\infty(\Om)}^{\frac{p}{p-\de}} |\Om|^{\frac{p(p^*_{s_1}-\de)}{(p-\de)p^*_{s_1}}}.$
Therefore, result follows from \eqref{eqb25} and \eqref{eqb26} with $C_\de=   \left(\frac{1}{\de}-\frac{1}{p^*_{s_1}}\right)A.$ \QED
 \end{proof}

 \begin{Lemma}\label{lemm3}
  (Palais-Smale range)Let $r= p^*_{s_1}$ then $\mc {J}_\la$ satisfies the $(PS)_c$ condition with $c$ satisfying 
  \begin{equation*}
     -\infty <c< c_{\infty}:= \frac{s_1}{n} S^{\frac{n}{ps_1}}-C_\de \la^{\frac{p}{p-\de}},
   \end{equation*}
 where $C_\de$ is the positive constant defined in \eqref{eqb27}.
 \end{Lemma}

\begin{proof}
\noi Let $\{u_k\}$ be a $(PS)_c$ sequence of $\mc{J}_\la$ in $X_{p,s_1}$. Therefore 
\begin{equation}\label{eqb28}
\begin{aligned}
   & \frac{1}{p}\|u_k\|^{p}_{X_{p,s_1}}+ \frac{\ba}{q}\|u_k\|^{q}_{X_{q,s_2}} -\frac{\la}{\de}\int_{\Om}a(x)|u_k|^{\de} dx -\frac{1}{p^*_{s_1}}\int_{\Om}  |u_k|^{p^*_{s_1}}dx=c+o_k(1)\\&
   \text{and  }\quad \|u_k\|^{p}_{X_{p,s_1}}+ \ba \|u_k\|^{q}_{X_{q,s_2}} -\la\int_{\Om}a(x)|u_k|^{\de} dx -\int_{\Om} |u_k|^{p^*_{s_1}} dx=o_k(1).
 \end{aligned}
\end{equation}
\noi Since $\{u_k\}$ is bounded in $X_{p,s_1}$, it implies there exists $u\in X_{p,s_1}$ such that up to subsequence $u_k \rightharpoonup u$ weakly in $X_{p,s_1}$ and $u$ is a critical point of $\mc{J}_\la$.\\
\textbf{Claim:} $u_k \ra u$ strongly in $X_{p,s_1}$.\\
Since $u_k \ra u$ strongly in $L^{\ga}(\Om)$ for $1\leq \ga < p^{*}_{s_1},$ it implies $$\ds\int_{\Om}a(x)|u_k|^{\de}dx \ra \ds\int_{\Om}a(x)|u|^{\de}dx \quad\text{ as } k\ra\infty.$$
And by Brezis-Lieb Lemma, we have
\begin{equation}\label{eqb29}
\begin{aligned}
 \|u_k\|^{p}_{X_{p,s_1}}=& \|u_k-u\|^{p}_{X_{p,s_1}}+\|u\|^{p}_{X_{p,s_1}}+ o_k(1),\quad \|u_k\|^{q}_{X_{q,s_2}}=\|u_k-u\|^{q}_{X_{q,s_2}}+\|u\|^{q}_{X_{q,s_2}}+ o_k(1),  \\&   \text { and }\int_{\Om}|u_k|^{p^*_{s_1}}= \int_{\Om}|u_k-u|^{p^*_{s_1}}+\int_{\Om}|u|^{p^*_{s_1}} + o_k(1).
\end{aligned}
\end{equation}
Coupling \eqref{eqb29} with \eqref{eqb28}, we obtain
\begin{equation}\label{eqb48}
\begin{aligned}
   & \frac{1}{p}\|u_k-u\|^{p}_{X_{p,s_1}}+ \frac{\ba}{q}\|u_k-u\|^{q}_{X_{q,s_2}}-\frac{1}{p^*_{s_1}}\int_{\Om}|u_k-u|^{p^*_{s_1}}=c-\mc {J}_\la(u)+ o_k(1)\\ \text { and }\qquad & \|u_k-u\|^{p}_{X_{p,s_1}}+ \ba \|u_k-u\|^{q}_{X_{q,s_2}}-\int_{\Om}|u_k-u|^{p^*_{s_1}}= o_k(1).
 \end{aligned}
\end{equation}
\noi Hence, let $\|u_k-u\|^{p}_{X_{p,s_1}}+\ba \|u_k-u\|^{q}_{X_{q,s_2}} \ra l$ and  $\ds\int_{\Om}|u_k-u|^{p^*_{s_1}}\ra l \quad \text{ as } k \ra \infty $.
If $l=0$, then claim is proved, so we assume $l>0$, then 
\begin{align*}
   l^{\frac{p}{p^*_{s_1}}}& =\left(\lim_{k \ra \infty}\int_{\Om}|u_k-u|^{p^*_{s_1}}dx \right)^{\frac{p}{p^*_{s_1}}} \\&
   \leq \lim_{k \ra \infty}\left( S^{-1}\|u_k-u\|_{X_{p,s_1}}^{p}\right)\\& \leq  S^{-1}\lim_{k \ra \infty}\left(\|u_k-u\|^{p}_{X_{p,s_1}}+ \ba \|u_k-u\|^{q}_{X_{q,s_2}} \right)=S^{-1} l.
\end{align*}
This implies $l \geq S^{\frac{n}{ps_1}}$.
\noi Now from \eqref{eqb48}, we deduce that
\begin{align*}
  c-\mc {J}_\la(u)\geq \frac{1}{p}\left(\|u_k-u\|^{p}_{X_{p,s_1}}+ \ba \|u_k-u\|^{q}_{X_{q,s_2}}\right)-\frac{1}{p^*_{s_1}}\int_{\Om}b(x)|u_k-u|^{p^*_{s_1}}+o_k(1)= \frac{s_1 \; l}{n}
\end{align*}
that is, $c\geq \frac{s_1l}{n}+\mc {J}_\la(u)\geq\frac{s_1}{n} S^{\frac{n}{ps_1}}-C_0 \la^{\frac{p}{p-\de}},$  hence we get a contradiction to $c<c_\infty$.
\end{proof}\QED
\textbf{Proof of Theorem \ref{pqthm2}} \textit{(i)}:
 Let $\ga_0>0$ be such that for all $\la\in(0,\ga_0)$ \begin{equation}\label{eqbp}
c_\infty =\frac{s_1}{n}S^{\frac{n}{ps_1}}-C_\de \la^{\frac{p}{p-\de}} >0\quad \mbox{and } \La_0 =\min \{ \ga_0, \la_0\}.
\end{equation}
 Using proposition \ref{propb2}, there exists a minimizing sequence $\{u_k\}$ in $\mc{N}_\la$ which is also a $(PS)_{\theta_\la}$ sequence for $\mc J_\la$. Employing Lemmas \ref{L35} and \ref{lemm3}, there exists $u_\la\in X_{p,s_1}$ such that $u_k\rightarrow u_\la$ strongly in $X_{p,s_1}$ for $\la\in(0, \La_0)$. Therefore for $\la\in(0, \La_0)$,\;$u_\la$ is a weak solution of problem $(\mc P_\la)$. 
As a consequence
\begin{align*}
  \mc J_\la(u_k)&= \left(\frac{1}{p}-\frac{1}{p^*_{s_1}}\right)\|u_k\|_{X_{p,s_1}}^{p} +\ba \left(\frac{1}{q}-\frac{1}{p^*_{s_1}}\right)\|u_k\|_{X_{q,s_2}}^{q}-\la\frac{p^*_{s_1}-\de}{\de p_{s_1}^*}\int_\Omega a(x)|u_k|^{\de}dx\\
  &\ge -\la\left(\frac{1}{\de}-\frac{1}{p^*_{s_1}}\right)\int_\Omega a(x)|u_k|^{\de}dx,	
\end{align*}
which gives us $$\ds\int_\Om a(x)|u_\la|^qdx \ge - \frac{p^*_{s_1}\de}{(p^*_{s_1}-\de)\la} \theta_{\la}>0.$$
Hence, $u_\la\not\equiv 0$. Thus, $u_\la\in \mc {N}_\la$ and $\mc {J}_\la(u_\la)=\theta_\la.$
Next we prove that $u_\la\in\mc N_\la^+$. On the contrary, let $u_\la\in\mc N_\la^-$, then by fibering map analysis there exist $t_1<t_2=1$ such that $t_1u_\la\in\mc N_\la^+$ and $t_2u_\la\in\mc N_\la^-$.
Since $\psi_{u_\la}$ is increasing in $[t_1,t_2)$, it implies
$$\theta_\la\leq \mc J_\la(t_1u_\la)<\mc J_\la(tu_\la)\le\mc J_\la(u_\la)=\theta_\la$$
for $t\in(t_1,1)$, which is a contradiction. Hence $u_\la\in\mc N_\la^+$ and $\theta_\la=\mc J_\la(u_\la)=\theta_\la^+$. Moreover, by using same assertions and arguments as in proof of Theorem \ref{pqthm1}, we obtain that  $u_\la$ is nonnegative solution.\QED
To prove  Theorem \ref{pqthm2} (ii), we will show   the existence of second solution below the first critical level by using the blowup analysis. In order to achieve this. we use the asymptotic estimate on the minimizers of $S$, which were proved by  Brasco et al. \cite{brasco}. Precisely, they  proved that $S$ has a minimizer and for every minimizer $U$, there exist $x_0\in\mb R^n$ and a constant sign monotone function $u:\mb R\ra \mb R$ such that $U(x)=u(|x-x_0|)$. We fix a radially symmetric, non-negative and decreasing minimizer $U(x)=U(|x|)$ of $S$. 
\begin{Lemma} 
	There exist $c_1, c_2>0$ and $\theta>1$ such that for all $r>1$, we have \begin{align*}
		\frac{c_1}{r^\frac{n-ps_1}{p-1}}\le U(r)\le \frac{c_2}{r^\frac{n-ps_1}{p-1}}, \quad \frac{U(\theta r)}{U(r)}\leq \frac{1}{2}.
	\end{align*}
\end{Lemma}
Multiplying by a positive constant, if necessary, we may assume 
\begin{equation}\label{eqb37}
	(-\De)_{p}^{s_1}U =U^{p^{*}_{s_1}-1}.
\end{equation}
We note that for any $\e >0$, $U_\e(x)= \e^{-\frac{n-ps_1}{p}} U\big(\frac{|x|}{\e}\big)$ is also a minimizer for $S$ satisfying \eqref{eqb37}. Without loss of generality assume $0\in\Om$ and as in \cite{brasco}, for $\e,\kappa >0$, we define the following functions\\
$m_{\e,\kappa}= \frac{U_\e(\kappa)}{U_\e(\kappa)-U_\e(\theta\kappa)},$  $g_{\e,\kappa}: [0,\infty)\to \mb R$ as $$ g_{\e,\kappa}(t)= \left\{
\begin{array}{lr}
0 \qquad\qquad\qquad\qquad\mbox{if} \ 0\le t\leq U_\e(\theta\kappa),\\
m_{\e,\kappa}^{p}(t-U_\e(\theta\kappa)) \  \quad\mbox{if} \ U_\e(\theta\kappa)\le t\leq U_\e(\kappa),\\
t+U_\e(m_{\e,\kappa}^{p-1}-1) \quad \mbox{if} \ t\ge U_\e(\kappa),
\end{array}
\right.$$
and $G_{\e,\kappa}: [0,\infty)\to \mb R$ by $$G_{\e,\kappa} =\int_0^t (g_{\e,\kappa}^\prime(s))^{\frac{1}{p}}ds =\left\{
\begin{array}{lr}
0 \qquad\qquad\qquad\qquad\mbox{if} \ 0\le t\leq U_\e(\theta\kappa),\\
m_{\e,\kappa}(t-U_\e(\theta\kappa)) \  \quad\mbox{if} \ U_\e(\theta\kappa)\le t\leq U_\e(\kappa),\\
t+U_\e(m_{\e,\kappa}^{p-1}-1) \quad \mbox{if} \ t\ge U_\e(\kappa).
\end{array} 
\right. $$
Define $u_{\e,\kappa}: [0,\infty)\to \mb R$ as   $u_{\e,\kappa}(r)= G_{\e,\kappa}(U_\e(r))$, a radially symmetric and non-increasing function, which satisfies 
\begin{equation}\label{auxfn}
u_{\e,\kappa}(r)= \left\{
\begin{array}{lr}
U_\e(r)\quad\mbox{if}\quad r\leq \kappa,\\
0\quad\quad\quad\mbox{if}\quad r\ge\theta\kappa.\\
\end{array}
\right.
\end{equation} 
 We will now state a Lemma which gives bounds on $\|u_{\e,\kappa}\|_{X_{p,s_1}}$ and $\|u_{\e,\kappa}\|_{L^{p^*_{s_1}}(\mathbb{R}^n)}$. 
\begin{Lemma}\label{lemm1}
\cite{mosconi} For every $\kappa$ and $0<\e\le\frac{\kappa}{2}$, there exists a constant $C= C(n,p,s_1)>0$  such that 
	\begin{align*}
	& \ds\int_{\mb R^{2n}}\frac{|u_{\e,\kappa}(x)-u_{\e,\kappa}(y)|^{p}}{|x-y|^{n+ps_1}}dxdy \leq S^{\frac{n}{ps_1}}+O\left(\Big(\frac{\e}{\kappa}\Big)^{\frac{n-ps_1}{p-1}}\right) \text{ and } \\&
    \ds\int_{\mb R^n}|u_{\e,\kappa}(x)|^{p^*_{s_1}}\geq   S^{\frac{n}{ps_1}}-C\Big(\frac{\e}{\kappa}\Big)^{\frac{n}{p-1}}.
	\end{align*}
\end{Lemma}
\begin{Lemma}\label{lemm2}
	There exists $\La_{00}$ such that for every $\la\in(0,\La_{00})$, there exists $u\geq 0$ in $X_{p,s_1}$ such that $$ \ds\sup_{t\geq 0} \mc{J}_\la(tu) <c_\infty .$$
	In particular $\theta_{\la}^{-} <c_\infty $.
\end{Lemma}
\begin{proof} Let $\La_0>0$ be as defined in \eqref{eqbp} so that for all $\la\in(0,\La_0)$, $c_\infty>0$ holds. Then,
	\begin{align*}
       \mc{J}_\la(tu_{\e,\kappa})&\leq \frac{t^{p}}{p}\|u_{\e,\kappa}\|_{X_{p,s_1}}^{p}+ \ba \frac{t^{q}}{q}\|u_{\e,\kappa}\|_{X_{q,s_2}}^{q} \\ &\leq 
       \frac{t^{p}}{p}\|u_{\e,\kappa}\|_{X_{p,s_1}}^{p}+C\ba \frac{t^{q}}{q}\|u_{\e,\kappa}\|_{X_{p,s_1}}^{p}  \leq
        C(t^{p}+t^{q}).
	\end{align*}
 Therefore, there exists $t_0\in(0,1)$ such that $$\ds\sup_{0\leq t\leq t_0}\mc{J}_\la(tu_{\e,\kappa}) < c_\infty.$$
 Let $h(t)= \frac{t^{p}}{p}\|u_{\e,\kappa}\|_{X_{p,s_1}}^{p} +\ba \frac{t^{q}}{q}\|u_{\e,\kappa}\|_{X_{q,s_2}}^{q}-\frac{t^{p^*_{s_1}}}{p^*_{s_1}}\ds\int_{\Om}|u_{\e,\kappa}|^{p^*_{s_1}},$ then
 we note that $h(0)=0$, $h(t)>0$ for $t$ small enough, $h(t)<0$ for $t$ large enough,
 and there exists $t_\e >0$ such that $\ds\sup_{t>0}h(t)=h(t_\e),$ that is 
   $$0=h^\prime(t_\e)=t_\e^{p-1}\|u_{\e,\kappa}\|_{X_{p,s_1}}^{p} +\ba t_\e^{q-1}\|u_{\e,\kappa}\|_{X_{q,s_2}}^{q}-t_\e^{p^*_{s_1}-1}\ds\int_{\Om}|u_{\e,\kappa}|^{p^*_{s_1}}$$
	which gives us
	\begin{align*}
	   t_\e ^{p^*_{s_1}-q}&=\frac{1}{\|u_{\e,\kappa}\|_{p^*_{s_1}}^{p^*_{s_1}}} \big(t_\e^{p-q}\|u_{\e,\kappa}\|_{X_{p,s_1}}^{p}+ \ba\|u_{\e,\kappa}\|_{X_{q,s_2}}^{q}\big)
	   < C(1+t_\e^{p-q}).
	\end{align*}
	Since $p^*_{s_1}> p$, there exists $t_1> 0$ such that $t_\e <t_1$ for all  $\e >0$.
	Also, for $\kappa>0$ such that $\theta \kappa<\kappa_1$, we have
	\begin{align*}
       \ds\int_\Om a(x)|u_{\e,\kappa}|^\de dx &=\ds\int_{B_{\theta\kappa}(0)}a(x)|u_{\e,\kappa}|^\de dx
	   \ge m_a\ds\int_{B_{\theta\kappa}(0)}|u_{\e,\kappa}|^\de dx 
	   \ge m_a\ds\int_{B_{\kappa}(0)}|U_{\e}|^\de dx.
		\end{align*}
	Therefore  we  obtain
	\begin{equation}\label{eqb31}
		\begin{aligned}
	      \ds\sup_{t\ge t_0}\mc{J_\la}(tu_{\e,\kappa})&\le \ds\sup_{t>0}h(t)-\frac{t_0^qm_a}{\de}\la\ds\int_{B_\kappa(0)}|U_\e|^\de\\
	      &=\frac{t_\e^{p}}{p}\|u_{\e,\kappa}\|_{X_{p,s_1}}^{p}+ \ba \frac{t_\e^{q}}{q}\|u_{\e,\kappa}\|_{X_{q,s_2}}^{q} -\frac{t_\e^{p^*_{s_1}}}{p^*_{s_1}}\|u_{\e,\kappa}\|_{L^{p^*_{s_1}}(\Om)}^{{p^*_{s_1}}}- \frac{t_0^{\de}}{\de}m_a\la\ds\int_{B_\kappa(0)}|U_\e|^\de\\
        &\leq \ds\sup_{t\ge 0}\left(\frac{t^{p}}{p}\|u_{\e,\kappa}\|_{X_{p,s_1}}^{p}- \frac{t^{p^*_{s_1}}}{p^*_{s_1}}\|u_{\e,\kappa}\|_{L^{p^*_{s_1}}(\Om)}^{p^*_{s_1}}\right)\\
        & \qquad \qquad +\ba \frac{t_1^{q}}{q}\|u_{\e,\kappa}\|_{X_{q,s_2}}^{q}  -\frac{t_0^{\de}}{\de}m_a\la\ds\int_{B_\kappa(0)}|U_\e|^\de.
	\end{aligned}
	\end{equation}
	Let $g(t)=\frac{t^{p}}{p}\|u_{\e,\kappa}\|_{X_{p,s_1}}^{p}- \frac{t^{p^*_{s_1}}}{p^*_{s_1}}\|u_{\e,\kappa}\|_{L^{p^*_{s_1}}(\Om)}^{p^*_{s_1}}$. A simple computation shows that $g$ attains maximum at $\tilde{t}=\left(\frac{\|u_{\e,\kappa}\|_{X_{p,s_1}}^{p}}{\|u_{\e,\kappa}\|_{L^{p^*_{s_1}}(\Om)}^{p^*_{s_1}}}\right)^\frac{1}{p^*_{s_1}-p}$. 
   Therefore, \begin{align*}
	    \ds\sup_{t\ge 0}g(t)= g(\tilde{t})= \frac{s_1}{n}	\left(\frac{\|u_{\e,\kappa}\|_{X_{p,s_1}}^{p}}{\|u_{\e,\kappa}\|_{L^{p^*_{s_1}}(\Om)}^{p^*_{s_1}}}\right)^\frac{n}{ps_1},
	 \end{align*}
	  using the estimates of Lemma \ref{lemm1}, we have
	\begin{align*}
	    \ds\sup_{t\ge 0}g(t)&\leq \frac{s_1}{n}\left(\frac{S^{\frac{n}{ps_1}}+O((\frac{\e}{\kappa})^{\frac{n-ps_1}{p-1}}) }{S^{\frac{n}{ps_1}}-C(\frac{\e}{\kappa})^{\frac{n}{p-1}}}\right)^{\frac{n}{ps_1}} \leq \frac{s_1}{n}S^{\frac{n}{ps_1}}+O\left(\Big(\frac{\e}{\kappa}\Big)^{\frac{n-ps_1}{p-1}}\right)
	\end{align*}
	hence, from \eqref{eqb31}, we deduce that
	\begin{align*}
	     \ds\sup_{t\ge t_0}\mc{J}_\la(tu_{\e,\kappa})\leq \frac{s_1}{n}S^{\frac{n}{ps_1}}+O\left(\Big(\frac{\e}{\kappa}\Big)^{\frac{n-ps_1}{p-1}}\right)+\ba \frac{t_1^{q}}{q}\|u_{\e,\kappa}\|_{X_{q,s_2}}^{q}- \frac{t_0^{\de}}{\de}m_a\la\ds\int_{B_\kappa(0)}|U_\e|^\de dx.
	\end{align*}
	Let $\ba=\e^\al$ with $\al>\frac{n-ps_1}{p-1}$, then we have
	  \begin{equation}\label{eqb32}
	  \begin{aligned}
           \ds\sup_{t\ge t_0}\mc{J}_\la(tu_{\e,\kappa})\leq \frac{s_1}{n}S^{\frac{n}{ps_1}}+C_1\e^{\frac{n-ps_1}{p-1}}-\frac{t_0^{\de}}{\de}m_a\la\ds\int_{B_\kappa(0)}|U_\e|^\de dx.
	  \end{aligned}
	 \end{equation}
	Next we estimate $\ds\int_{B_\kappa(0)}|U_\e|^\de dx$.
	For $\kappa>0$, sufficiently small such that $B_{\theta\kappa}(0)\Subset\Om$, $\theta\kappa <\kappa_1$ and $0<\e<\kappa/2$, we have 
	 \begin{align*}
	     \ds\int_{B_\kappa(0)}|U_\e|^qdx &=\e^{n-\frac{n-ps_1}{p}\de}\ds\int_{B_{\frac{\kappa}{\e}}(0)}|U|^\de dx\\ &\geq
    	\e^{n-\frac{n-ps_1}{p}\de}\omega_{n-1}\ds\int_{1}^{\frac{\kappa}{\e}}|U(r)|^\de r^{n-1}dx\\ &\geq
	    \e^{n-\frac{n-ps_1}{p}\de}\omega_{n-1}c_1^\de\ds\int_{1}^{\frac{\kappa}{\e}}r^{n-\frac{n-ps_1}{p}\de-1}dx\\ &\simeq 
	    C_3\left\{
	    \begin{array}{lr}
	     \e^{n-\frac{n-ps_1}{p}\de}  \qquad\qquad \ \mbox{if}\quad \de>\frac{n(p-1)}{n-ps_1},\\
	     \e^{n-\frac{n-ps_1}{p}\de}|\log\e| \qquad\mbox{if}\quad \de=\frac{n(p-1)}{n-ps_1}.\\
	   \end{array}
	  \right.
	\end{align*}
	Choosing $\e=(\la^\frac{p}{p-\de})^\frac{p-1}{n-ps_1}$, \eqref{eqb32} yields 
	\begin{align*}
	  \ds\sup_{t\ge t_0}\mc{J}_\la(tu_{\e,\kappa}) &\leq \frac{s_1}{n}S^{\frac{n}{ps_1}}+C_1\la^{\frac{p}{p-\de}}\\	& \quad- C_3\la\left\{
	  \begin{array}{lr}
	   \la^{\frac{p(p-1)}{(p-\de)(n-ps_1)} \left(n-\frac{n-ps_1}{p}\de\right)} \quad \mbox{if } \ \de>\frac{n(p-1)}{n-ps_1},\\
	   \la^\frac{\de}{p-\de} |\log\la^\frac{p\de}{(p-\de)n}|  \qquad\quad\mbox{ if } \ \de=\frac{n(p-1)}{n-ps_1}.\\
	\end{array}
	\right.
	\end{align*}
\textbf{Case 1:} If $\de>\frac{n(p-1)}{n-ps_1}$.\\
Clearly
\begin{align*}
  1&+\frac{p(p-1)}{(p-\de)(n-ps_1)} \left(n-\frac{n-ps_1}{p}\de\right)
     < \frac{p}{p-\de}	
\end{align*}
 if and only if $\de>\frac{n(p-1)}{n-ps_1}$. 
 Thus, there exists $\ga_1>0$ such that for all $\la\in(0,\ga_1)$,
   $$C_1\la^\frac{p}{p-\de}-C_3\;\la^{1+{\frac{p(p-1)}{(p-\de)(n-ps_1)}  \left(n-\frac{n-ps_1}{p}\de\right)}}< -C_\de\la^\frac{p}{p-\de}.$$
	
\textbf{Case 2:} If $\de=\frac{n(p-1)}{n-ps_1}$\\
Since $|\log\la^\frac{p\de}{n(p-\de)}|\ra\infty$ as $\la\ra 0$, therefore we can choose $\ga_2>0$ such that 
$$ C_1\la^\frac{p}{p-\de}-C_3\;\la^\frac{p}{p-\de}|\log\la^\frac{p\de}{n(p-\de)}|< - C_\de\la^\frac{p}{p-\de}.$$
Let $\La_{00}=\min\{\La_0,\ga_1, \ga_2, (\frac{\kappa}{2})^\frac{n-ps_1}{p-1}\} >0$, $\ba_{00}=\La_{00}^\frac{p}{p-\de}$ and $\e_{00}=(\La_{00}^\frac{p}{p-\de})^\frac{p-1}{n-ps_1}$. Then, for all $\la\in(0,\La_{00})$, $\ba\in(0, \ba_{00})$ and $\e\in(0,\e_{00})$, we obtain
$$\ds\sup_{t\ge 0}\mc{J}_\la(tu_{\e,\kappa})< c_\infty.$$
Now choosing $\kappa>0$ sufficiently small we see that $u_{\e,\kappa}\in X_{p,s_1}$ and by fibering map analysis there exists $\hat{t}>0$ such that $\hat{t}u_{\e,\kappa}\in\mc N_\la^-$.
Hence, \begin{align*}
\theta_\la^- \leq\mc J_\la(\hat{t}u_{\e,\kappa})\leq \ds\sup_{t\ge 0}\mc J_\la(tu_{\e,\kappa})< c_\infty.
\end{align*} 
 This completes the proof of Lemma.  \QED
\end{proof}	
{\textbf{Proof of Theorem \ref{pqthm2}}} \textit{(ii)}:
Lemma \ref{tt} and Proposition \ref{propb2} hold true even when $\mc N_\la$ is replaced by $\mc N_\la^-$, hence we get a minimizing sequence $\{u_k\} \subset\mc N_\la^-$ such that $\mc J_\la(u_k)=\theta_\la^- +o_k(1)$ and $\mc J^\prime_\la(u_k)=o_k(1)$ that is, $\{u_k\}$ is a $(PS)_{\theta_\la^-}$ sequence for $\mc J_\la$. From Lemmas \ref{lemm3} and \ref{lemm2}, there exists $v_\la\in X_{p,s_1}$ such that $u_k\ra v_\la$ in $X_{p,s_1}$ and Theorem \ref{thb2} ensures $\langle \mc J_\la^\prime(v_\la), v_\la\rangle =0$. Now, using strong convergence $u_k\ra v_\la$ and noting that $\mc N_\la^0 =\emptyset$, we get $v_\la\in\mc N_\la^-$ and $\theta_\la^- =\mc J_\la(v_\la)$. Using similar argument as in proof of Theorem \ref{pqthm1} we can show that $v_\la$ is non-negative.\QED

\begin{Remark} If $s_1=s_2$, then the result of Lemma \ref{cheeg} does not hold true. In this case, the energy functional $\mc J_\la$ is well defined on the space $X:=X_{p,s_1}\cap X_{q,s_2}$. Then one can defined 
	the Nehari manifold associated to $\mc J_\la$ as 
	 \begin{equation*}
	  \mc {N_\la}=\{{u\in X\setminus\{0\}}: \langle \mc {J_\la}^\prime(u), u\rangle =0\},
	\end{equation*}
	where $\ds\langle ._,. \rangle$ is the duality between $X$ and its dual space.
	Observe that $\mc {N_\la} \not\equiv \emptyset$. One can easily prove that the functional $\mc J_\la$ is coercive and bounded below on $\mc {N_\la}$. 
	 Using fibering map analysis and Nehari manifold technique (See section 3) we can obtain the existence and multiplicity results as in Theorem \ref{pqthm1} and Theorem \ref{pqthm2}\textit{(i)} on the space $X$. 
\end{Remark}

\section {The case $\de=q$}
In this section we prove existence  of an solution of the problem $(\mc P_\la)$ for the case $\de=q\leq p<r\leq p^*_{s_1}$ by using the blowup analysis and mountain pass theorem. In this case problem $(\mc P_\la)$ reduces to 
	 \begin{equation}\label{probq}
	   \left\{
	     \begin{array}{rlll}\ds
	      \quad  (-\Delta)^{s_1}_{p}u+ \ba (-\Delta)^{s_2}_{q}u &=& \la a(x)|u|^{q-2}u+ b(x)|u|^{r-2} u \quad \text{in} \; \Om \\
	      \quad \quad  u &=&0\quad\quad \text{on} \; \mb R^n\setminus \Om,
	      \end{array}
	    \right.
	\end{equation}
 The energy functional associated to problem \eqref{probq}, $\mc I_\la:X_{p,s_1}\to \mb R$ is  defined as 
  \begin{align*}
	  \mc I_\la(u) :=\frac{1}{p}\|u\|_{X_{p,s_1}}^{p}+ \frac{\ba}{q}\|u\|_{X_{q,s_2}}^{q} -\frac{\la}{q}\int_{\Om}a(x)|u|^{q} ~dx -\frac{1}{r}\int_{\Om}b(x)|u|^{r} ~dx.
  \end{align*}
 \begin{Lemma}\label{mpg}
   Suppose $1<q\le p<r\le p^*_{s_1}$, then there exists $\la_*>0$ such that 
      \begin{enumerate}
		\item[(i)] there exist $\rho,\eta>0$ such that for all $\la\in(0,\la_*)$, $$\mc I_\la(u)>\eta\; \text{ for all}\; u\in X_{p,s_1}\; \text{ with} \; \|u\|_{X_{p,s_1}}=\rho,$$ 
		\item[(ii)]  there exists $u_0\in X_{p,s_1}$ such that $\mc I_\la(u_0)<0$ and $\|u_0\|_{X_{p,s_1}}>\rho$.
	  \end{enumerate}
 \end{Lemma}
  \begin{proof}
  By means of H\"older inequality and Sobolev embeddings, we have 
	  \begin{equation}\label{eqb61}
	   \begin{aligned}
        \mc I_\la(u) &\ge \frac{1}{p}\|u\|_{X_{p,s_1}}^{p}+ \frac{\ba}{q}\|u\|_{X_{q,s_2}}^{q}-\la\frac{C_1}{q}\|u\|_{X_{p,s_1}}^q-\frac{C_2}{r}\|u\|_{X_{p,s_1}}^r \\
        &\ge  \frac{1}{p}\|u\|_{X_{p,s_1}}^{p}-\la\frac{C_1}{q}\|u\|_{X_{p,s_1}}^q-\frac{C_2}{r}\|u\|_{X_{p,s_1}}^r \\
        &=\|u\|_{X_{p,s_1}}^q \left( \frac{1}{p}\|u\|_{X_{p,s_1}}^{p-q}-\frac{C_2}{r}\|u\|_{X_{p,s_1}}^{r-q}-\la\frac{C_1}{q}\right).
	   \end{aligned}
	 \end{equation}
   Let $f:\mb R^+\to\mb R$ be defined as $f(t)=\frac{1}{p}t^{p-q}-\frac{C_2}{r}t^{r-q}$. A simple calculation shows that $f$ attains maximum at $t_0 =\left(\frac{(p-q)r}{p\;C_2(r-q)}\right)^\frac{1}{r-p}$ and $f(t_0)=\frac{(r-p)}{p(r-q)}\left(\frac{(p-q)r}{p\;C_2(r-q)}\right)^\frac{p-q}{r-p}$. Now choose $\la_*>0$ such that  \[\eta:= f(t_0)-\la_* \frac{C_1}{q}>0,\] then for all $\la\in(0,\la_*)$, we have $f(t_0)-\la\frac{C_1}{q}>\eta$. Therefore, choosing $\rho=t_0$,  \eqref{eqb61} implies 
   \begin{align*}
   	\mc I_\la(u)>\eta \quad \mbox{for all } u\in X_{p,s_1} \ \mbox{with } \|u\|_{X_{p,s_1}}=\rho.
   \end{align*}
   Since $\mc I_\la(tu)\ra -\infty$ as $t\ra\infty$, we can choose $\tilde{t}>0$ such that $\|\tilde{t}u\|_{X_{p,s_1}}>\rho$ and $\mc I_\la(\tilde{t}u)<0$, then setting $u_0=\tilde{t}u$ implies $(ii).$\QED
  \end{proof}
Define 
   \[ c_\la :=\inf\bigg\{\ds\sup_{t>0} \mc I_\la(tu) :u\in X_{p,s_1}\setminus\{0\}\bigg\}.\]

On the same line of proof of Lemmas \ref{lemmab3} and \ref{lemm3}, we can prove the following results:


\begin{Lemma}\label{lemm62}
  \begin{enumerate}
	\item[(i)] $\mc I_\la$ satisfies $(PS)_c$ condition for all $c\in\mb R$, provided $r<p^*_{s_1}$,
	\item[(ii)] $\mc I_\la$ satisfies $(PS)_c$ condition for all $c$ satisfying   
		 \begin{equation*}
		   -\infty <c< c_{\infty}:= \frac{s_1}{n} S^{\frac{n}{ps_1}}-C_q\; \la^{\frac{p}{p-q}},
		 \end{equation*}
	where
		 \begin{equation*}
		   C_q=\left(\frac{(p^*_{s_1}-q)(p-q)}{p \; q \;p^*_{s_1}}\right) \left(\frac{p^*_{s_1}-q}{p^*_{s_1}-p}\right)^{\frac{pq}{p-q}} S^{\frac{-q}{p-q}}\|a\|_{L^\infty(\Om)}^{\frac{p}{p-q}} \; |\Om|^{\frac{p(p^*_{s_1}-q)}{(p-q)p^*_{s_1}}},
		   \ \mbox{ provided }r=p^*_{s_1}.
		 \end{equation*} 
	 \end{enumerate}
  \end{Lemma}

\begin{Lemma}\label{lemm63}
 Let $1<\frac{n(p-1)}{n-ps_1}\le q\le p$. There exists $\La_{*}$ such that for every $\la\in(0,\La_{*})$, there exists $\ba_{*}>0$ such that for all $\ba\in(0,\ba_{*})$ 
 $$ c_\la <c_\infty .$$
\end{Lemma}
\begin{proof}
  Following the proof of Lemma \ref{lemm2} with $u_{\e,\kappa}$, define in \eqref{auxfn}, there exists $\La_{00}>0$ and $\e_{00}$ such that for $\la\in(0,\La_{00})$, there exists $\ba_{00}>0$ such that for all $\la\in(0,\La_{00})$, $\e\in(0,\e_{00})$ and $\ba\in(0,\ba_{00})$, 
	\[\ds\sup_{t\ge 0} \mc I_\la(tu_{\e,\kappa})<c_\infty.\]
  Let $\La_{*}=\min\{\La_{00}, \la_*\}$ and analogously define $\e_*>0$ and $\ba_*>0$. Then, for all $\la\in(0,\La_{*})$, $\e\in(0,\e_{*})$ and $\ba\in(0,\ba_{*})$, 
	\[\ds\sup_{t\ge 0} \mc I_\la(tu_{\e,\kappa})<c_\infty.\]
 Hence, ${c_\la}<c_\infty$, for all $\la\in(0,\La_{*})$.\QED
\end{proof}
\textbf{Proof of Theorem \ref{pqthm3}}: Using  Lemmas \ref{mpg},  \ref{lemm62}, \ref{lemm63} and standard Mountain Pass Theorem, there exists  a solution  $w_\la\in X_{p,s_1}$  of  \eqref{probq}. \QED

\end{document}